\documentclass[a4paper]{amsart}
\usepackage{amssymb}
\usepackage{verbatim}
\usepackage{amscd}
\def\card{{{\operatorname{card}}}}

\numberwithin{equation}{section}
\theoremstyle{plain}
\newtheorem{theorem}[equation]{Theorem}
\newtheorem{corollary}[equation]{Corollary}

\newtheorem{lemma}[equation]{Lemma}
\newtheorem*{(DQ1)}{(DQ1)}
\theoremstyle{definition}

\theoremstyle{remark}

\begin{document}
\title [Markov-Dyck shifts ]{ Markov-Dyck shifts, neutral periodic points and topological conjugacy}
\author{Wolfgang Krieger}
\author{Kengo Matsumoto}
\begin{abstract}
We study the  neutral periodic points of  Markov-Dyck shifts 
of finite strongly connected directed graphs.
Under certain hypothesis on the structure of the graphs $G$ we show, that the topological conjugacy of their Markov-Dyck shifts $M\negthinspace {\scriptstyle D}({G})$ implies the isomorphism of the graphs.
\end{abstract}
\maketitle

\section{Introduction}
Let $\Sigma$ be a finite alphabet, and let $S $ be the left shift on
$\Sigma^{\Bbb Z}$,
$$
S((x_i)_{ i \in {\Bbb Z}})   = (x_{i+1})_{ i \in {\Bbb Z}}, \qquad  
(x_i)_{ i \in {\Bbb Z}} \in \Sigma^{\Bbb Z}.
$$
The closed shift-invariant subsystems of the shifts $S$ are called
subshifts. For an introduction to the theory of subshifts see \cite{Ki} and \cite{LM}. 
A finite word in the symbols of 
$\Sigma$  is called admissible for the subshift 
$X \subset \Sigma^{\Bbb Z} $ if it
appears somewhere in a point of $X$. 
A subshift $X \subset \Sigma^{\Bbb Z} $ is uniquely determined by its language $\mathcal L(X)$
of admissible words. 

In this paper we study the topological conjugacy  of subshifts that are constructed from
finite
directed graphs. We denote a finite directed graph $G$ with
vertex set ${\mathcal V}$ and edge set 
${\mathcal E}$ by $G(\mathcal V  , \mathcal E  )$. 
The source vertex of an edge $e \in {\mathcal E}$ or of a directed path in the graph we denote by
$s$ and its target vertex by $t$ (or  by $s_G$ and  $t_G$, in the case that we have to distinguish between graphs.) 
We consider strongly connected finite directed graphs $G = G(\mathcal V  , \mathcal E  ) $. It is assumed that $G$ is not a cycle.
We recall the construction of the Markov-Dyck shift $M\negthinspace {\scriptstyle D}({G})$ of $G$ (see \cite {M1}).
Let 
$
{\mathcal E}^- = \{e^-: e \in   \mathcal E  \}
$
be a copy of ${\mathcal E}$.
Reverse the directions of the edges in ${\mathcal E}^-$
to obtain the 
edge set
$
{\mathcal E}^+ = \{e^+: e \in   \mathcal E  \}
$
of the reversed graph 
of $G(\mathcal V  , \mathcal E^-)  $.
In this way one has defined a graph $\widetilde{G}( \mathcal V  , {\mathcal E}^- \cup{\mathcal E}^+  )$, where
\begin{align*}
&s_{\widetilde{G}}(e^-) =s_G(e)     , \quad  t_{\widetilde{G}}(e^-) = t_G(e)    ,
\\
&s_{\widetilde{G}}(e^+) =t_G(e)  , \quad t_{\widetilde{G}}(e^+) = s_G(e) , \qquad e \in   \mathcal E .
\end{align*}
With idempotents $\bold 1_V, V \in {\mathcal V},$ 
the set 
${\mathcal E}^- \cup \{\bold 1_V:V \in {\mathcal V}\}\cup {\mathcal E}^+$
is the generating set of the graph inverse semigroup $\mathcal S(G)$ of $G$ (see \cite [Section 10.7]{L}), where, besides
$\bold 1_V^2 = \bold 1_V, V \in {\mathcal V},$ the relations are
$$
\bold 1_U\bold 1_W = 0, \qquad   U, W \in {\mathcal V}, U \neq W,  
$$
$$
f^-g^+ =
\begin{cases}
\bold 1_{s_G(f)}, &\text{if  $f = g$}, \\
0, &\text {if  $f \neq g$},\quad f,g \in {\mathcal E},
\end{cases}
$$
\begin{equation*}
\bold 1_{s_G(f)} f^- = f^- \bold 1_{t_G(f)},\qquad
\bold 1_{t_G(f)} f^+ = f^+ \bold 1_{s_G(f)},\qquad
f \in {\mathcal E}. 
\end{equation*}
The subsemigroup of the semigroup $\mathcal S(G)$, that is generated by $\mathcal E^-$ ($\mathcal E^+$)
we denote by $\mathcal S^-(G)$($\mathcal S^+(G)$), and we refer to the elements of $\mathcal E^-$($\mathcal E^+$) as the generators of $\mathcal S^-(G)$($\mathcal S^+(G)$).

The alphabet of 
$M\negthinspace {\scriptstyle D}({G})$
 is
${\mathcal E}^-\cup {\mathcal E}^+$,
and a word 
$(e_k)_{1 \leq k \leq K}$ 
is admissible for 
$M\negthinspace {\scriptstyle D}({G})$
 precisely if 
$$
\prod_{1 \leq k \leq K}e_k \neq 0.
$$
The directed graphs with a single vertex and $N>1$
loops yield the Dyck inverse monoids (the "polycycliques" of \cite{NP}) $\mathcal D_N$ and the Dyck shifts $D_N$ \cite {Kr1}.

Given a  subshift $X \subset \Sigma^{\Bbb Z}$ we
set
$$
x_{[i,j]}= (x_{k})_{i\leq k \leq j},   
$$
and
$$
X_{[i,j]}
=
\{ x_{[i,j]} : x\in X \}, \quad i,j \in \Bbb Z, i \leq j, \qquad  x \in X, 
$$
and we  use similar notation in the case that indices range in semi-infinite intervals.
Set
$$
\Gamma^+_X(a) = \{ x^+ \in X_{(j, \infty)}: a x^+ \in X_{[i, \infty)}\}, \quad 
a \in  X_{[i,j]}, \quad i,j \in \Bbb Z, i \leq j.
$$
The notation $\Gamma^-$ has the symmetric meaning.
Also set
$$
\omega^+_X(a) = \bigcap_{x^-\in \Gamma^-(a)} \{ x^+ \in \Gamma^+(a):x^-a x^+ \in X\}, \quad
 a \in  X_{[i,j]}, \quad i,j \in \Bbb Z, i \leq j.
$$
The notation $\omega^-$ has the symmetric meaning.
Also set 
$$
 A_n(X) =  
\bigcap _{i\in \Bbb Z}
( \{x  \in X: 
x_{[i,\infty)} \in \omega_X^{+} (x_{[i - n, i)})\}\cap \{x  \in X: 
x_{(-\infty,i]} \in \omega_X^{-} (x_{(i,i + n]})\}),
$$
and
$$
 A(X) = \\
\bigcup_{n \in \Bbb N} A_n(X).
$$
The periodic points in $A(X)$ are called the neutral periodic points of $X$.
In Section 2 we 
clarify the structure of the set of neutral periodic points of a Markov-Dyck shift. 
This includes 
a characterization of  the neutral periodic points of the shift.

For a finite directed graph $G( \mathcal V , \mathcal E )$ 
we denote  by $\mathcal F_G$ the set of  edges that are the single incoming edges of their target  vertices
and we denote by 
$ \mathcal R_G$ the set of $V \in   \mathcal V_G $ that have more than one incoming edge. 
The set $ \mathcal R_G$ is the set of roots of a set of (possibly degenerate) directed rooted trees. \footnote{The graph with one vertex and no edges is generally considered to be a tree. We refer to this graph as the degenerate directed tree.}
We denote the vertex set of the directed  tree with root  $R\in  \mathcal R_G$ by $ \mathcal V_ R$,
and its edge set by $\mathcal F_R$. One has
$$
 \mathcal V = \bigcup_{R \in  \mathcal R_G } \mathcal V_ R ,  \qquad \mathcal F_G = \bigcup_{R \in  \mathcal R_G } \mathcal F_R.
$$
The condition, that $\card  ( \mathcal R_G ) = 1$, is equivalent to the condition, that the graph 
$G(  \mathcal V , \mathcal F_G)$ is a (possibly degenerate) directed tree.
We denote the graph that is obtained by  contracting the non-degenerate trees among the trees 
$G( \mathcal V_ R , \mathcal F_R ), R \in  \mathcal R_G,  $
  to their roots $ R  $ by 
$\widehat{G}( \mathcal R_G , \mathcal E\setminus \mathcal F_G)$. In the graph $\widehat{G}$ the source vertex of an edge $ e \in\mathcal E\setminus \mathcal F_G$  is the root of the tree that has $s_G(e)$ as a leave, and its target vertex is
$t_G(e)$. 
In \cite {Kr2} a Property $(A)$ of subshifts, an invariant of topological conjugacy, was introduced,  and to a subshift $X$ with Property $(A)$  a semigroup $\mathcal S(X)$ was invariantly associated. In 
\cite {HK} it was shown that the Markov-Dyck shift 
$M\negthinspace{\scriptstyle {D}}(G(\mathcal V, \mathcal E))$ of a graph  
$G(\mathcal V, \mathcal E)$ %
 has Property $(A)$, and that 
 $$
 \mathcal S(M\negthinspace {\scriptstyle {D}}
 (G( \mathcal V  , \mathcal E  )))= \mathcal S( \widehat{G}(  \mathcal R_G , \mathcal E\setminus \mathcal F_G) ).
 $$
 In Section 2 we  show that a topological conjugacy of Markov-Dyck shifts of graphs
  $G( \mathcal V  , \mathcal E  )$
 induces an isomorphism of the graphs 
  $\widehat{G}(  \mathcal R_G , \mathcal E\setminus \mathcal F_G)   $, 
  that also preserves certain data pertaining to the configuration of the neutral periodic points of 
  the Markov-Dyck shift.
For Markov-Motzkin shifts (see \cite [Section 4.1]{KM2}) analogous results hold.
       
In Section 3 we 
consider finite directed graphs $G(\mathcal V, \mathcal E)$, such that $\card (\mathcal R_G) = 1$. In this case, following the terminology, that was introduced in \cite {HI}, we say that a  periodic point $p$ of 
$M\negthinspace {\scriptstyle D}({G})$ and its orbit
have negative  multiplier $e \in  \mathcal E  \setminus \mathcal F_G  $, if there exists an $i \in \Bbb Z$ and an $M\in \Bbb N$, such that
$$
 \lambda(p_{[i , i + \pi(p))}) =(\widehat {e}^-)^M.
$$ 
The mapping that assigns to a 
multiplier
$e\in \mathcal E \setminus \mathcal F_G$
the set of periodic points  of
$M\negthinspace {\scriptstyle D}({G})$
with negative 
multiplier
$e$
is an invariant of topological conjugacy \cite [Proposition 4.2]{HIK}. 
We set 
$$
\mathcal M(M\negthinspace {\scriptstyle D}({G})) =\mathcal E \setminus \mathcal F_G,\qquad
\nu(M\negthinspace {\scriptstyle D}({G})) = \card ( \mathcal M(M\negthinspace {\scriptstyle D}({G}))),
$$
We denote 
for a multiplier $e \in  \mathcal E  \setminus \mathcal F_G$, by 
$I^{(e)}_k(M\negthinspace {\scriptstyle D}({G}))$ the number of periodic points of 
$M\negthinspace {\scriptstyle D}({G})$ with negative multiplier $e$ and period $k$, and we set
\begin{align*}
\Lambda^{(e)}(M\negthinspace {\scriptstyle D}({G})) = \min \{ k \in \Bbb N: 
I^{(e)}_k(M\negthinspace {\scriptstyle D}({G})) > 0 \}.
 \end{align*} 
The set  
$\{ \Lambda^{(e)}(M\negthinspace {\scriptstyle D}({G})): e \in  \mathcal E  \setminus \mathcal F_G  \}$
is an invariant of topological conjugacy.
 We denote by $ I^{(0)}_{2k}(M\negthinspace {\scriptstyle D}({G})) $ the number of neutral periodic points of period $2k$ of $M\negthinspace {\scriptstyle D}({G}), k \in \Bbb N$.
We also denote by
$
\Xi^{(e)}_{2k}(M\negthinspace {\scriptstyle D}({G}))  
$
the number of orbits of length $2k, k\in \Bbb N$, with negative multiplier 
$e \in  \mathcal M(M\negthinspace {\scriptstyle D}({G}))$.

In Section 3 we 
consider three families, $\bold{F}_I,\bold{F}_{II}$ and $\bold{F}_{III}$ of finite directed graphs $G(\mathcal V, \mathcal E)$, such that $\card (\mathcal R_G) = 1$. 
For 
the graphs in 
each of these families we introduce canonical models. 
Each  canonical model is specified by a set of parameters, that we call the "data" of the model.
We then establish for the graphs $G$ in each of these families, that the topological conjugacy class of the Markov-Dyck shift of $G$ determines the isomorphism class of the graph $G$. The proof consists in showing, that the invariants 
 $ \nu(M\negthinspace {\scriptstyle D}({G})),$ and $\Lambda^{(e)}, 
e \in \mathcal M(M\negthinspace {\scriptstyle D}({G})) ,$ together with $I^{(e)}_k(M\negthinspace {\scriptstyle D}({G})),\Xi^{(e)}_{2k}, k \in \Bbb N, e \in  \mathcal M(M\negthinspace {\scriptstyle D}({G}))$ 
contain sufficient information to determine the "data" of the canonical model of the graph $G$. We also characterize the Markov-Dyck shifts of the graphs in each of these families within the Markov-Dyck shifts. 

The family $\bold{F}_I$ contains the finite strongly connected directed  graphs  
$G(\mathcal V, \mathcal E)$, such that $\card (\mathcal R_G) = 1$, 
and  
such that all vertices, except the root of the tree 
$G(\mathcal V ,\mathcal F_G)$, have out-degree one.

The family $\bold{F}_{II} $ contains the finite strongly connected directed  graphs $G(\mathcal V, \mathcal E)$ such that $\card (\mathcal R_G) = 1$,  
and such that  all leaves of the 
tree $G(\mathcal V,\mathcal F_G)$ are at level one.

The family $\bold{F}_{III}$ contains the finite strongly connected directed  graphs $G(\mathcal V, \mathcal E)$,  such that 
the graph $G(\mathcal V ,\mathcal F_G)$ is a tree, 
that has the shape of a "V",
and that is such that the two leaves of the tree $G(\mathcal V ,\mathcal F_G)$ have the same out-degree, 
and all interior vertices of the tree $G(\mathcal V ,\mathcal F_G)$ have out-degree one.
 
It had been known, that for finite directed graphs, in which every vertex has at least two incoming edges, topological conjugacy of their Markov-Dyck shifts implies the isomorphism of the graphs
(see \cite[Section 4]{HIK} and \cite [Corollary 3.2]{Kr3}).
Actually, for finite directed graphs, in which every vertex has at least two incoming edges,  the flow equivalence of their Markov-Dyck shifts implies the isomorphism of the graphs. For Dyck shifts this follows from \cite {M2}, and for the general case see \cite {CS} and \cite {Kr4}.

{\bf Acknowledgment.}
Thanks go to  the organizers of the research program "Classification of operator algebras: complexity, rigidity, and dynamics" at the Mittag-Leffler Institute, January - April 2016,
for the opportunity to work on this paper while both authors were  attending the research program.
Thanks go to Toshihiro Hamachi for discussions that lead to the improvement of Section 3.

\section{Neutral periodic points of Markov-Dyck shifts}

We denote the period of a periodic point of a subshift by $\pi$. 
Given a (non-degenerate) rooted tree $\mathcal T( \mathcal V, \mathcal F)$, we denote for $V \in  \mathcal V$ by $b(V)$ the path from the root to $V$.

We continue to consider a strongly connected finite  directed graph 
$G = G(\mathcal V,\mathcal E),$ such that
$ \card (\mathcal E\setminus \mathcal F_G)>1$.
An edge 
 $
e \in \mathcal E \setminus \mathcal F_G,
$
determines a  generator  of $\mathcal S^-(\widehat{G})$($\mathcal S^+(\widehat{G})$), that we denote by 
$ \widehat {e}^-$($ \widehat {e}^+$), and we set
$$
\lambda(e^-) = \widehat {e}^-, \qquad \lambda(e^+) = \widehat {e}^+, \qquad
e \in \mathcal E \setminus \mathcal F_G.
$$A root $R\in \mathcal R_G$ determines an idempotent of $\mathcal S(\widehat{G})$, that we denote by $\widehat{\bold 1}_R$, and
we set
$$
\lambda(f^-) = \lambda(f^+) = \widehat{\bold 1}_R, \qquad f\in \mathcal F_R, R\in \mathcal R_G.
$$

We set
$$
(e^-)^{-1}=e^+, \quad (e^+)^{-1}=e^-, \qquad e \in \mathcal E,
$$
and, more generally,
for a path $b=(b_i)_{1 \leq i \leq I} \in \mathcal L(M\negthinspace {\scriptstyle D}({G}))$
we use the notation
$$
b^{-1}= ((b_i)^{-1})_{I \geq i \geq 1}
$$ 
for the reversed path of $b$.

The set of neutral periodic points of $M\negthinspace {\scriptstyle D}({G})$ we denote by $P^{(0)}(M\negthinspace {\scriptstyle D}({G}))$.
 For a vertex $V \in\mathcal V,$
 we denote by $P^{(0)}(V)$ the set of periodic points $p$ of 
 $M\negthinspace {\scriptstyle D}({G})$, such that there is an
  $ m \in \Bbb Z,$ such that 
$$
V = s_{\widehat {G}}(p_{[m, m +\pi (p))}) ,
$$
and such that
\begin{align*}
 \bold 1_V=\prod _{m \leq j <  m +\pi (p)}p_j  .
\end{align*}

\begin{theorem} 
$$
P^{(0)}(M\negthinspace {\scriptstyle D}({G})) = \bigcup_{ V \in \mathcal V}
P^{(0)}(V).
$$
\end{theorem}
\begin{proof} 

For the proof  let there be given a neutral periodic point $p$ of 
$M\negthinspace {\scriptstyle D}({G})$.
With $\widehat {K}_+\in \Bbb Z_+,\widehat {K}_-\in \Bbb Z_+$, and
$$
\widehat {e}^+_{\widehat {k}(+)} \in\widehat { \mathcal E}^+,    \qquad \widehat {K}_+
\geq \widehat {k}(+)>0 ,
$$
$$
\widehat {e}^-_{\widehat {k}(-)} \in \widehat {\mathcal E}^-,
\qquad 0 <\widehat {k}(-)\leq \widehat {K}_-,
$$
and with an $  {R}\in  \mathcal R_G  $, write
$$
\prod _{0 \leq j <  \pi (p)}\lambda(p_j) =
( \prod _{\widehat {K}_+\geq \widehat {k}(+)>0}\widehat {e}^+_{\widehat {k}(+)}) \ 
 \widehat {\bold 1}_{ R} \
(  \prod _{0 <\widehat {k}(-)\leq \widehat {K}_-}\widehat {e}^-_{\widehat {k}(-)}).
$$
Note that
$$
s_{\widehat{G}}(\widehat {e}_{\widehat {K}_+}   ) = t_{\widehat{G}}(\widehat {e}_{\widehat {K}_-}   ).
$$
 Assume that
 \begin{align*}
 (  \prod _{0 <\widehat {k}(-)\leq \widehat {K}_-}\widehat {e}^-_{\widehat {k}(-)})
 ( \prod _{\widehat {K}_+\geq \widehat {k}(+)>0}\widehat {e}^+_{\widehat {k}(+)})
 \in \mathcal S^-( \widehat {G})
  \setminus 
  \{\widehat{\bold 1}_{R} \}. \tag {2.1}
 \end{align*}
Choose an $ \widehat{m}\in [0, \pi(p))  $, such that
$$
\prod _{0 \leq j <  \widehat{m}}\lambda(p_j)  
= \prod _{\widehat {K}_+\geq \widehat {k}(+)>0}
\widehat {e}^+_{\widehat {k}(+)}.
$$
Then also
$$
\prod _{\widehat{m} \leq j <  \pi (p)}\lambda(p_j) 
 = \prod _{0 <\widehat {k}(-)\leq \widehat {K}_-}
\widehat {e}^-_{\widehat {k}(-)},
$$
and therefore by (2.1)
$$
\prod _{\widehat{m} \leq j < \widehat{m}+\pi(p)}\lambda(p_j) 
 =  (  \prod _{0 <\widehat {k}(-)\leq \widehat {K}_-}\widehat {e}^-_{\widehat {k}(-)})
 ( \prod _{\widehat {K}_+\geq \widehat {k}(+)>0}\widehat {e}^+_{\widehat {k}(+)})
 \in    \mathcal S^-( \widehat {G}).
$$
It follows that the word $p_{[\widehat{m}, \widehat{m}+\pi(p)) } $ has a suffix $e^-b$, that is uniquely determined by the condition, that $ e^- \in  \mathcal E^- \setminus \mathcal F^-_G$ and that $b$ is  the empty word, or  $b =(b_i)_{1 \leq i \leq I}$ is a word in 
$\mathcal L(M\negthinspace {\scriptstyle D}({G}))$
 such that
 $$
 \prod_{1 \leq i \leq I} \lambda (b_i) = \bold 1_{t(e)}.
 $$
Let $g\neq e$ be an incoming edge of $t_{\widetilde{G}}(e^-)$. For $K \in \Bbb N$ the words 
$$
e^-bp_{[\widehat{m}, \widehat{m}+K\pi(p)) } 
$$ 
and 
$$
p_{[\widehat{m}, \widehat{m}+K\pi(p)) } p_{[\widehat{m}, \widehat{m}+K\pi(p)) }^{-1} b^{-1}g^+
$$
are admissible
for 
$M\negthinspace {\scriptstyle D}({G})$.
However the word
$$
e^-bp_{[\widehat{m}, \widehat{m}+K\pi(p)) } p_{[\widehat{m}, \widehat{m}+K\pi(p)) }^{-1} b^{-1}g^+
$$
is not admissible for 
$M\negthinspace {\scriptstyle D}({G})$. This contradicts the neutrality of $p$. Under the assumption that
  \begin{align*}
 (  \prod _{0 <\widehat {k}(-)\leq \widehat {K}_-}\widehat {e}^-_{\widehat {k}(-)})
 ( \prod _{\widehat {K}_+\geq \widehat {k}(+)>0}\widehat {e}^+_{\widehat {k}(+)})
 \in \mathcal S^+( \widehat {G})
  \setminus 
  \{\widehat{\bold 1}_{R} \},
 \end{align*}
 one has the symmetric argument. We have shown that
  \begin{align*}
 \lambda(p_{[\widehat{m}, \widehat{m}+\pi(p)) }  ) =\widehat{\bold 1}_{R}. \tag {2.2}
  \end{align*}

To repeat this reasoning, 
with $K_+ \in\Bbb Z_+,K_-\in \Bbb Z_+$, and
$$
e^+_{k(+)} \in \mathcal E^+,    \qquad K_+\geq k(+)>0 ,
$$
$$
e^-_{k(-)} \in \mathcal E^-,\qquad 0 <k(-)\leq K_-,
$$
 and with a $ {V}\in   {\mathcal V}  $, write
$$
\prod _{ \widehat{m} \leq j < \widehat{m}+  \pi (p)}p_j =
( \prod _{K_+\geq k(+)>0}e^+_{k(+)}) \ \bold 1_{ {V}} \thinspace
 ( \prod _{0 <k(-)\leq K_-}e^-_{k(-)}).
$$
Note that
$$
s_{\widetilde{G}}(  e^+_{K(+)} ) = t_{\widetilde{G}}(  e^-_{K(-)}  ).
$$
 Assume that
 \begin{align*}
 (  \prod _{0 < {k}(-)\leq {K}_-}e^-_{{k}(-)})
 ( \prod _{ {K}_+\geq  {k}(+)>0}e^+_{ {k}(+)})
 \in \mathcal S^-({G})
  \setminus 
  \{\bold 1_{{V}} \}. \tag {2.3}
 \end{align*}
Choose an ${m}\in [0, \pi(p))$, such that
$$
\prod _{ \widehat{m} \leq j <\widehat{m}+{m}} p_j
 = \prod _{ {K}_+\geq  {k}(+)>0}e^+_{ {k}(+)}.
$$
Then also
$$
\prod _{ \widehat{m}+{m} \leq j < \widehat{m}+ \pi(p)} p_j
 = \prod _{0 < {k}(-)\leq  {K}_-}e^-_{ {k}(-)},
$$
and therefore by (2.3)
$$
\prod _{m \leq j <m + \pi(p)} p_j
 =  (  \prod _{0 < {k}(-)\leq  {K}_-}e^-_{ {k}(-)})
 ( \prod _{ {K}_+\geq  {k}(+)>0}e^+_{ {k}(+)})
 \in \mathcal S^-(G).
$$
It follows from this and from (2.2),  that there is 
an $R \in \mathcal R_G$, such that there is 
a directed path $(f_l)_{1\leq l \leq L}, L\in \Bbb N,$ in  
 the tree $G( \mathcal V_R ,   \mathcal F_R )$,
 such that
$$
\prod _{m \leq j <m + \pi(p)} p_j= \prod _{1\leq l \leq L}f^-_l.
$$
This contradicts the periodicity of $p$. Under the assumption that
 \begin{align*}
 (  \prod _{0 < {k}(-)\leq {K}_-}e^-_{{k}(-)})
 ( \prod _{ {K}_+\geq  {k}(+)>0}e^+_{ {k}(+)})
 \in \mathcal S^+({G})
  \setminus 
  \{\bold 1_{{V}} \}. 
 \end{align*}
one has the symmetric argument. This confirms that
$$
 \bold 1_V=\prod _{m \leq j <  m +\pi (p)}p_j ,
$$
and completes the proof.
\end{proof}

The set of neutral periodic points  of a subshift $X \subset \Sigma^\Bbb Z$ carries a pre-order relation $\lessapprox(X)$ (see \cite {Kr2}). For neutral periodic points $q$ and $r$ of $X$ one has $q  \gtrapprox(X) r$, if there exists a point in $A(X)$, that is left asymptotic to the orbit of  $q$ and right 
 asymptotic to the orbit of  $r$.
The equivalence relation that is derived from  $\lessapprox(X)$ we denote by  
$\approx\negthinspace \negthinspace(X)$.

We set
$$
P^{(0)}_R = \bigcup_{V \in \mathcal V_R:}P^{(0)}(V), \qquad R\in  \mathcal R_G.
$$
The proof of the following lemma is similar to the proof of Theorem 3.2 of \cite {HK}.
 
\begin{lemma}
The $\approx\negthinspace \negthinspace(X)$-equivalence class of   $p \in  P^{(0)}_R  ,
 R \in  \mathcal R_G,$ coincides with $P^{(0)}_R$.
\end{lemma}
\begin{proof} 
The proof comes in two parts.
For the first part
let $R \in  \mathcal R_G$, let $U, W\in  \mathcal V_R,$ and
$q, r \in P^{(0)}_R$, and let $j, k \in \Bbb Z, $ be such that
$$
U = s_{\widetilde G}(p_{[j, j +\pi (q))}) , \qquad
W = s_{\widetilde G}(r_{[k, k +\pi (r))}) ,
$$
and such that
\begin{align*}
 \bold 1_U=\prod _{j \leq i < j +\pi (p)}q_i  , \qquad
 \bold 1_W=\prod _{k \leq i < k +\pi (r)}r_i  .
\end{align*}
Let a point $x \in M\negthinspace {\scriptstyle D}({G}))$ be given by
$$
x_{(-\infty, 0]}= q_{(-\infty, j]}b^{-1}(U), \qquad
x_{(0,\infty)}=b(W)r_{(k,\infty)}.
$$
One has
$$
x \in A_{\max \{\pi(q), \pi(r)\}}(M\negthinspace {\scriptstyle D}({G})).
$$
This follows since the edges in the paths $b(U)$ and $b(W)$ are by construction the only incoming edges of their target vertices. We have proved that
$
q\approx\negthinspace \negthinspace(X)r.
$

For the second part 
let $R, R^\prime \in \mathcal R_G,$
\begin{align*}
R \neq R^\prime, \tag {2.4}
\end{align*}
 let 
 $$
 V \in \mathcal V_R, p\in P^{(0)}_R,\quad V^\prime \in  \mathcal V_{R^\prime}, p^\prime\in P^{(0)}_{R^\prime},
 $$
  and let $j, j^\prime \in \Bbb Z $ be such that
$$
V = s_{\widetilde{G}}(p_{[j, j +\pi (p))}) ,\quad V^\prime = s_{\widetilde{G}}
(p^\prime_{[j^\prime, j^\prime+\pi (p^\prime))}) ,
$$
and
\begin{align*}
 \bold 1_V=\prod _{j \leq i<  j +\pi (p)}p_i  , \quad 
 \bold 1_{V^\prime}=\prod _{j^\prime \leq i^\prime<  j^\prime +\pi (p^\prime)}p^\prime_{i^\prime}.
\end{align*}
We prove that  $p$ and $p^\prime$ are 
$ \lessapprox \negthinspace(M\negthinspace {\scriptstyle D}({G}))$-incomparable. Assume that 
$$
p  \lessapprox\negthinspace(M\negthinspace {\scriptstyle D}({G})) \ p^\prime,
$$
 and let
 $J\in \Bbb N$, and 
\begin{align*}
x \in A_J(M\negthinspace {\scriptstyle D}({G}))), \tag {2.5} 
\end{align*}
and  $m, m^\prime \in \Bbb Z, m < m^\prime,$ be  such that
$$
x_{(\infty, m)} =   p_{(\infty, j) },\qquad x_{[m^\prime  ,   \infty)} =
 p^\prime_{[m^\prime  ,   \infty)} .
$$
With $K, K^\prime\in \Bbb Z_+, $ and 
$$
e_{k}\in \mathcal E \setminus \mathcal F_G,\qquad K \geq k > 0,
$$
$$
e^\prime_{k^\prime}\in \mathcal E \setminus \mathcal F_G,\qquad 
0<  k^\prime \leq   K^\prime,
$$
and with $Q \in \mathcal R_G$,
$$
t_G(e_K  ) = Q = t_G(e_{K^\prime}),
$$
 write
$$
\prod _{m \leq i <  m^\prime}\lambda(x_i) = 
(\prod_{K \geq k > 0}\widehat {e}^+_k) \
 \widehat {\bold 1}_{ Q} \
(  \prod _{0 < {k^\prime}\leq  {K^\prime}}\widehat {e^{\prime}}^-_{ {k^\prime}}).
$$
Assumption (2.4) 
implies that $(K, K^{\prime}) \neq (0,0).$ Assume $K > 0.$
Then we have an $M\in [ m, m^{\prime})$, such that
$
x_M = e_K^+,
$
and
$$
\prod_{m\leq i < M}\lambda (x_i) =  \widehat {\bold 1}_{ Q} .
$$
Let $g\neq e_K,$ be an incoming edge of $Q$. The word
$$
g^-b(V)^-p_{[ m - J\pi(p) ,   m)}x_{[ m  ,   M)}
$$
is admissible for $M_D(G)$. However the word
$$
g^-b(V)^-p_{[ m - J\pi(p) ,   m)}x_{[ m  ,   M)}e_K^+
$$
is not. This contradicts (2.5).
Under the assumption that $K^\prime >0,$ one has the symmetric argument. We have shown that $p\not \lessapprox \negthinspace\negthinspace(M\negthinspace {\scriptstyle D}({G})))\thinspace p^\prime$.
 \end{proof}

Denoting by
$\Pi_n(Y)$ the number of points of period $n$ of a shift-invariant set $Y
\subset \Sigma^{\Bbb Z} ,$
the zeta function of $Y$ 
is given by
$$
\zeta_Y(z) = e^{\sum_{n \in {\Bbb N}} \frac{\Pi_n(X)z^n}{n}}.
$$

For the finite strongly connected directed graph $G( \mathcal V  , \mathcal E )$,  we  vertex weigh the graph
$
\widehat{G}({ \mathcal R_G} , \mathcal E\setminus \mathcal F_G )
$
by assigning to its vertices $R \in  \mathcal R_G$ the zeta function of $P^{(0)}_R.$

\begin{corollary} 
For finite strongly connected directed graphs $G( \mathcal V  , \mathcal E )$  the  topological  conjugacy of the Markov-Dyck shifts  $M\negthinspace {\scriptstyle D}({G})$  implies  the  isomorphism  of the 
vertex weighted  graphs 
$
\widehat{G}({ \mathcal R_G} , \mathcal E\setminus \mathcal F_G )
$
 with weights $(\zeta_{P^{(0)}_R})_{R\in \mathcal R_G}$.
\end{corollary}
\begin{proof} 
There is a canonical projection $\chi$ of the set of points in 
$M\negthinspace {\scriptstyle D}({G})$, that are left-asymptotic to a point in 
$P^{(0)}(M\negthinspace {\scriptstyle D}({G}))$, and also right-asymptotic to point in 
$P^{(0)}(M\negthinspace {\scriptstyle D}({G}))$  onto the associated semigroup
$$
 \mathcal S(M\negthinspace {\scriptstyle D}({G}))= \mathcal S( \widehat{G}(  \mathcal R_G , \mathcal E\setminus \mathcal F_G) ).
 $$
(compare \cite [Section 3] {Kr3} and \cite [Section 3] {Kr2}). A topological conjugacy induces an isomorphism of the associated semigroups and it acts accordingly on the  inverse images under $\chi$ of each of the elements in the set
$$
\{ \bold 1_R  : R \in  \mathcal R_G \} \cup \{  e^- : e \in  \mathcal E \setminus \mathcal F_G  \} \subset \mathcal S( \widehat{G}(  \mathcal R_G , \mathcal E\setminus \mathcal F_G) ),
$$
(see \cite [Section 2]{Kr5}).
By Lemma 2.2 
$$
\chi^{-1}(\bold 1_R) = P_R^{(0)}, \qquad  R \in \mathcal R_G,
$$
and
$$
\bigcup_{R \in \mathcal R_G}\chi^{-1}(\bold 1_R) = P^{(0)}(M\negthinspace {\scriptstyle D}({G})).\qed
$$
\renewcommand{\qedsymbol}{}
\end{proof}

To a vertex $V\in  \mathcal V$ we associate the circular code $\mathcal C_V$ that contains the  words  $(c_i)_{1 \leq i \leq I} \in \mathcal L(M\negthinspace {\scriptstyle D}({G}))$
such that 
$$
s_{\widetilde {G}}( c )  =  t_{\widetilde {G}}( c ) = V, 
$$
and 
$$
\prod_{1\leq i \leq I}c_i = \bold 1_V,
$$
$$
\prod_{1\leq i \leq J}c_i \neq  \bold 1_V, \qquad 1 < J < I.   
$$
The generating function of $\mathcal C_V$ we denote by $\varphi_V$.

\begin{corollary}
For $G( \mathcal V, \mathcal E)$,
$$
\zeta_{P^{(0)}(M\negthinspace {\scriptscriptstyle D}({G}))} = \prod_{V \in \mathcal V} \frac {1}{1- \varphi_V}.
$$
\end{corollary}
\begin{proof}
The corollary follows from Theorem 2.1 (e.g. see \cite[Section 5]{P} or \cite [Section 2]{KM1}).
\end{proof}

\section{Families of finite directed graphs }

In this section we consider the case of
strongly connected finite  directed graphs
$G=G(\mathcal V, \mathcal E),$ such that $\card (\mathcal R_G) = 1$.
These graphs are precisely the graphs, that have a  Dyck inverse monoid as the  associated semigroup of their Markov-Dyck shifts.
Complete invariants for the isomorphism for these directed graphs are known for the case that all source vertices $s(e), e \in \mathcal M(M\negthinspace {\scriptstyle D}({G}))$,  have the same out-degree (see \cite {GM}).
We denote the root of $\mathcal F_G$ by $V_0$, and the out-degree of $V_0$ by $D(V_0)$.
 We set  
$$
\mathcal M_\ell(M\negthinspace {\scriptstyle D}({G})) = \{e \in \mathcal M(M\negthinspace {\scriptstyle D}({G})): \Lambda^{(e)}  = \ell  \}, \qquad 
\ell \in \Bbb N.
$$
By the use of the notation 
$\Lambda(M\negthinspace {\scriptstyle D}({G})))$ we indicate that all lengths 
$  \Lambda^{(e)}(M\negthinspace {\scriptstyle D}({G})), e \in  \mathcal E  \setminus \mathcal F_G,  $
are equal and that $\Lambda(M\negthinspace {\scriptstyle D}({G})))$  is their common value.

\subsection{A family of finite directed graphs I}
Set
$$
\Pi _I= \{(S_\ell  )_{\ell \in \Bbb N}\in \Bbb Z_+^\Bbb N:  1<
 \sum_{\ell \in \Bbb N}S_\ell < \infty \}.
$$
The data $(S_\ell  )_{\ell \in \Bbb N}  \in \Pi_I$ determine  canonical models 
$G((S_\ell  )_{\ell \in \Bbb N})$ of the graphs in $\bold{F}_{I}$. We define
$G((S_\ell  )_{\ell \in \Bbb N})$ as the graph with vertices $V_0$ and 
$$
V_{\ell, s, t}, \qquad 1\leq t < \ell, 1 \leq s \leq S_\ell, \ \ell >1,
$$
and edges
$$
f_{\ell, s, t}, \qquad 1\leq t < \ell, 1 \leq s \leq S_\ell, \ \ell >1,
$$
 and
 \begin{align*}
e_{\ell, s}, \qquad  \  \  \ 1 \leq s \leq S_\ell, \quad \ell\in \Bbb N.
 \end{align*}
 The source and target mappings are given by
 \begin{align*}
&s(f_{\ell, s,  1}   ) = V_0,
\\
&s(f_{\ell, s,  t}   ) = V_{\ell, s,  t-1} , \qquad 1 < t <\ell,
\\
&t(f_{\ell, s, t}  )= V_{\ell, s, t}, \qquad  \ \ \ 1\leq t < \ell,\quad 0< s \leq S_\ell, \ \ell >1,
 \end{align*}
and
$$
s(e_{1,s}   )= V_{0} = t(e_{1,s}   ), \qquad     0< s \leq S_1, 
$$
\begin{align*}
s(e_{\ell, s}) &=  V_{\ell,s,\ell -1},  t(e_{\ell, s})= V_0, \qquad  0< s \leq S_\ell, \ \ell >1.
\end{align*}
One has
$$
\mathcal F_{G((S_\ell  )_{\ell \in \Bbb N})} = \{f_{\ell, s, t}: 1\leq t < \ell, 1 \leq s \leq S_\ell, \ \ell >1\}.
$$
The Dyck shifts $D_N, N> 1,$ belong here with the data $S_1= N, S_\ell = 0, \ell > 1$.
Also the Fibonacci-Dyck shift belongs here with the data $S_1 = 1, S_2 = 1, S_\ell = 0, \ell > 2 $.  
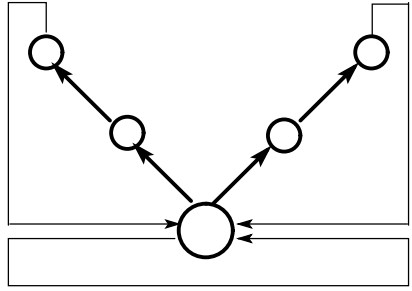
\begin{figure}[htbp]
\begin{center}
{
\unitlength 0.1in
\begin{picture}( 20.7200, 14.7700)( 20.5000,-24.6700)
%
\special{pn 20}%
\special{ar 3072 2180 140 140  0.0000000 6.2831853}%
%
\special{pn 20}%
\special{pa 3072 2040}%
\special{pa 3072 2040}%
\special{fp}%
%
\special{pn 20}%
\special{pa 3110 2040}%
\special{pa 3390 1760}%
\special{fp}%
\special{sh 1}%
\special{pa 3390 1760}%
\special{pa 3330 1794}%
\special{pa 3352 1798}%
\special{pa 3358 1822}%
\special{pa 3390 1760}%
\special{fp}%
%
\special{pn 8}%
\special{pa 4116 2146}%
\special{pa 3240 2146}%
\special{pa 3240 2146}%
\special{fp}%
%
\special{pn 8}%
\special{pa 3262 2146}%
\special{pa 3240 2146}%
\special{fp}%
\special{sh 1}%
\special{pa 3240 2146}%
\special{pa 3308 2166}%
\special{pa 3294 2146}%
\special{pa 3308 2126}%
\special{pa 3240 2146}%
\special{fp}%
%
\special{pn 20}%
\special{pa 2996 2026}%
\special{pa 2716 1746}%
\special{fp}%
\special{sh 1}%
\special{pa 2716 1746}%
\special{pa 2748 1808}%
\special{pa 2754 1784}%
\special{pa 2776 1780}%
\special{pa 2716 1746}%
\special{fp}%
%
\special{pn 20}%
\special{ar 3478 1684 84 84  0.0000000 6.2831853}%
%
\special{pn 20}%
\special{ar 2666 1670 84 84  0.0000000 6.2831853}%
%
\special{pn 20}%
\special{pa 3562 1606}%
\special{pa 3842 1326}%
\special{fp}%
\special{sh 1}%
\special{pa 3842 1326}%
\special{pa 3782 1360}%
\special{pa 3804 1364}%
\special{pa 3810 1388}%
\special{pa 3842 1326}%
\special{fp}%
%
\special{pn 20}%
\special{ar 3930 1250 84 84  0.0000000 6.2831853}%
%
\special{pn 20}%
\special{pa 2576 1606}%
\special{pa 2296 1326}%
\special{fp}%
\special{sh 1}%
\special{pa 2296 1326}%
\special{pa 2328 1388}%
\special{pa 2334 1364}%
\special{pa 2356 1360}%
\special{pa 2296 1326}%
\special{fp}%
%
\special{pn 20}%
\special{ar 2246 1250 84 84  0.0000000 6.2831853}%
%
\special{pn 8}%
\special{pa 4122 990}%
\special{pa 4122 2152}%
\special{fp}%
%
\special{pn 8}%
\special{pa 3934 1158}%
\special{pa 3934 998}%
\special{pa 4122 998}%
\special{pa 4122 998}%
\special{pa 4122 998}%
\special{fp}%
%
\special{pn 8}%
\special{pa 2058 2146}%
\special{pa 2932 2146}%
\special{pa 2932 2146}%
\special{fp}%
%
\special{pn 8}%
\special{pa 2910 2146}%
\special{pa 2932 2146}%
\special{fp}%
\special{sh 1}%
\special{pa 2932 2146}%
\special{pa 2866 2126}%
\special{pa 2880 2146}%
\special{pa 2866 2166}%
\special{pa 2932 2146}%
\special{fp}%
%
\special{pn 8}%
\special{pa 2050 990}%
\special{pa 2050 2152}%
\special{fp}%
%
\special{pn 8}%
\special{pa 2246 1144}%
\special{pa 2246 990}%
\special{pa 2050 990}%
\special{pa 2050 990}%
\special{pa 2050 990}%
\special{fp}%
%
\special{pn 8}%
\special{pa 2912 2222}%
\special{pa 2050 2222}%
\special{pa 2050 2468}%
\special{pa 4122 2468}%
\special{pa 4122 2222}%
\special{pa 4122 2222}%
\special{pa 4122 2222}%
\special{fp}%
%
\special{pn 8}%
\special{pa 4122 2222}%
\special{pa 3254 2222}%
\special{pa 3240 2222}%
\special{pa 3240 2222}%
\special{pa 3240 2222}%
\special{fp}%
%
\special{pn 8}%
\special{pa 3268 2222}%
\special{pa 3240 2222}%
\special{fp}%
\special{sh 1}%
\special{pa 3240 2222}%
\special{pa 3308 2242}%
\special{pa 3294 2222}%
\special{pa 3308 2202}%
\special{pa 3240 2222}%
\special{fp}%
\end{picture}%
}
\end{center}
\caption{$G(1,0,2,0,\dots)$}
\end{figure}

\begin{theorem}
For a finite directed graph $G=G(\mathcal V, \mathcal E)$ there exist data
$$
(S_\ell  )_{\ell \in \Bbb N} \in \Pi_I,
$$
such that there is a topological conjugacy 
\begin{align*}
M\negthinspace {\scriptstyle D}({G}) \simeq M\negthinspace {\scriptstyle D}(G((S_\ell  )_{\ell \in \Bbb N})), \tag {3.I.1}
\end{align*}
if and only if the associated semigroup of $M\negthinspace {\scriptstyle D}({G})$ is a Dyck inverse monoid, 
and
\begin{align*}
\tfrac {1}{2}I^{(0)}_2(M\negthinspace {\scriptstyle D}({G})) =
 \nu (M\negthinspace {\scriptstyle D}({G})) + \sum_{\ell >1}(\ell - 1)
 \card( \mathcal M_\ell(M\negthinspace {\scriptstyle D}({G})), 
\end{align*}
and in this case (3.I.1) holds for
\begin{align*}
S_\ell = \card(\mathcal M_\ell(M\negthinspace {\scriptstyle D}({G})   ), \qquad \ell \in \Bbb N. \tag {3.I.2}
\end{align*}
\end{theorem}
\begin{proof}
The statement holds if $G$ is a single vertex graph. Consider a graph $\widetilde{G} = G(\widetilde{\mathcal V}, \widetilde{\mathcal E})$,
such that the graph $G( \widetilde{\mathcal V} , \mathcal F_{\widetilde{G}}  )$ is a non-degenerate tree. 
Denote by $J_{\widetilde{G}}(\ell)$ the number of leafs of 
$\mathcal F_{\widetilde{G}}$ that have level $\ell -1$. One has that
\begin{align*}
\card (\mathcal F_{\widetilde{G}} ) \leq \sum_{\ell >1}(\ell - 1) J_{\widetilde{G}}(\ell) , \tag {3.I.3}
\end{align*}
and
\begin{align*}
J_{\widetilde{G}}(\ell)\leq \card (\mathcal M_\ell(M\negthinspace {\scriptstyle D}(\widetilde{G}))). \qquad \ell > 1,
\tag {3.I.4}
\end{align*}
Therefore
\begin{align*}
\card (\mathcal F_{\widetilde{G}} ) \leq \sum_{\ell >1}(\ell - 1)
 \card (\mathcal M_\ell(M\negthinspace {\scriptstyle D}(\widetilde{G}))) . \tag {3.I.5}
\end{align*}
Equality holds simultaneously in  
(3.I.3) and (3.I.4) 
if and only if it holds in (3.I.5)
if and only if  $G$ belongs to $\bold{F}_{I}$. Equation (3.I.2) implies equality in (3.I.5)
 and the theorem follows.
\end{proof}

\begin{corollary}
For finite directed graphs $G=G(\mathcal V, \mathcal E)$ such that
the associated semigroup of $M\negthinspace {\scriptstyle D}({G}) $ is a Dyck inverse monoid, and such that
\begin{align*}
\tfrac {1}{2}I^{(0)}(M\negthinspace {\scriptstyle D}&({G}) ) =
 \nu (M\negthinspace {\scriptstyle D}({G}) ) + \sum_{\ell >1}(\ell - 1)
 \card( \mathcal M_\ell(M\negthinspace {\scriptstyle D}({G}) ))
\end{align*}
 the topological conjugacy of their Markov-Dyck shifts implies the isomorphism of the graphs.
\end{corollary}
\begin{proof}
In (3.I.2) the data $ (S_\ell  )_{\ell \in \Bbb N}$ are expressed in terms of invariants of topological conjugacy.
\end{proof}

\subsection{A family of finite directed graphs II}

We set
$$
\Pi = \{(R, (Q_M)_{M \in \Bbb N}) \in \Bbb Z_+\times \Bbb  N^{\Bbb N}:
 1 <R +  \sum_{M\in\Bbb N }Q_M < \infty \}.
$$
The data $(R, (Q_M)_{M \in \Bbb N}) \in \Pi$ determine  canonical models 
$G((R, (Q_M)_{M \in \Bbb N}))$ of the graphs in $\bold{F}_{II}$. We define
$G((R, (Q_M)_{M \in \Bbb N}))$ 
as the directed graph with a vertex $V(0)$ and with vertices
$$
V_{M, q}(1), \qquad 1\leq q\leq Q_M, \quad M \in \Bbb N, 
$$
and with edges
$$
f_{M, q}, \qquad 1\leq q\leq Q_M, \quad M \in \Bbb N, 
$$
and
$$
e_r, \qquad 1 \leq r \leq R,
$$
and
$$
e_{M, q,m}, \qquad 1\leq m\leq M, \ 1\leq q\leq Q_M, \quad M \in \Bbb N.
$$
The source and target mappings are given by
$$
s(f_{M, q}   ) = V(0),
\ \
t(f_{M, q}   ) = V_{M, q}(1),\qquad 1\leq q\leq Q_M, \quad M \in \Bbb N,
$$
and
$$
s(e_r) =  t(e_r)  = V(0),  \quad 1 \leq r \leq R,
$$
and
$$
s(e_{M, q,m} ) = V_{M, q}, \ \
t(e_{M, q,m}) =  V(0),\qquad 1\leq m\leq M,1\leq q\leq Q_M, \quad M \in \Bbb N.
$$

One has
 $$
 \mathcal F_{G(  R, (Q_M)_{M \in \Bbb N})} = \{f_{M, q}, 1\leq q\leq Q_M,  M \in \Bbb N\}.
 $$ 
 Note the non-empty intersection of $\bold{F}_{II}$ with $\bold{F}_{I}$.
\begin{figure}[htbp]
\begin{center}
{
\unitlength 0.1in
\begin{picture}( 19.7400, 18.4100)( 19.8000,-29.1100)
%
\special{pn 20}%
\special{ar 3044 2772 140 140  0.0000000 6.2831853}%
%
\special{pn 20}%
\special{pa 3044 2632}%
\special{pa 3044 2632}%
\special{fp}%
%
\special{pn 20}%
\special{ar 3500 1770 84 84  0.0000000 6.2831853}%
%
\special{pn 20}%
\special{ar 3184 1770 84 84  0.0000000 6.2831853}%
%
\special{pn 20}%
\special{ar 2884 1778 84 84  0.0000000 6.2831853}%
%
\special{pn 8}%
\special{pa 2344 1070}%
\special{pa 2344 1070}%
\special{fp}%
%
\special{pn 20}%
\special{pa 2968 2624}%
\special{pa 2898 1890}%
\special{fp}%
\special{sh 1}%
\special{pa 2898 1890}%
\special{pa 2884 1958}%
\special{pa 2902 1942}%
\special{pa 2924 1954}%
\special{pa 2898 1890}%
\special{fp}%
%
\special{pn 20}%
\special{pa 3100 2624}%
\special{pa 3190 1892}%
\special{fp}%
\special{sh 1}%
\special{pa 3190 1892}%
\special{pa 3162 1956}%
\special{pa 3184 1944}%
\special{pa 3202 1960}%
\special{pa 3190 1892}%
\special{fp}%
%
\special{pn 20}%
\special{pa 3150 2646}%
\special{pa 3432 1868}%
\special{fp}%
\special{sh 1}%
\special{pa 3432 1868}%
\special{pa 3390 1924}%
\special{pa 3414 1918}%
\special{pa 3428 1938}%
\special{pa 3432 1868}%
\special{fp}%
%
\special{pn 8}%
\special{pa 3506 1890}%
\special{pa 3198 2688}%
\special{pa 3198 2688}%
\special{fp}%
%
\special{pn 8}%
\special{pa 3208 2662}%
\special{pa 3198 2688}%
\special{fp}%
\special{sh 1}%
\special{pa 3198 2688}%
\special{pa 3242 2632}%
\special{pa 3218 2638}%
\special{pa 3204 2618}%
\special{pa 3198 2688}%
\special{fp}%
%
\special{pn 8}%
\special{pa 3220 1686}%
\special{pa 3220 1582}%
\special{pa 3710 1582}%
\special{pa 3710 2736}%
\special{pa 3710 2736}%
\special{pa 3710 2736}%
\special{fp}%
%
\special{pn 8}%
\special{pa 3710 2736}%
\special{pa 3234 2736}%
\special{pa 3226 2736}%
\special{pa 3226 2736}%
\special{fp}%
%
\special{pn 8}%
\special{pa 3254 2736}%
\special{pa 3226 2736}%
\special{fp}%
\special{sh 1}%
\special{pa 3226 2736}%
\special{pa 3294 2756}%
\special{pa 3280 2736}%
\special{pa 3294 2716}%
\special{pa 3226 2736}%
\special{fp}%
%
\special{pn 8}%
\special{pa 3184 1672}%
\special{pa 3184 1442}%
\special{pa 3814 1442}%
\special{pa 3814 2806}%
\special{pa 3814 2806}%
\special{pa 3814 2806}%
\special{fp}%
%
\special{pn 8}%
\special{pa 3814 2806}%
\special{pa 3234 2806}%
\special{pa 3234 2806}%
\special{pa 3234 2806}%
\special{fp}%
%
\special{pn 8}%
\special{pa 3262 2806}%
\special{pa 3234 2806}%
\special{fp}%
\special{sh 1}%
\special{pa 3234 2806}%
\special{pa 3300 2826}%
\special{pa 3286 2806}%
\special{pa 3300 2786}%
\special{pa 3234 2806}%
\special{fp}%
%
\special{pn 8}%
\special{pa 3150 1680}%
\special{pa 3150 1266}%
\special{pa 3954 1266}%
\special{pa 3954 2876}%
\special{pa 3954 2876}%
\special{pa 3954 2876}%
\special{fp}%
%
\special{pn 8}%
\special{pa 3954 2876}%
\special{pa 3234 2876}%
\special{pa 3234 2876}%
\special{pa 3234 2876}%
\special{fp}%
%
\special{pn 8}%
\special{pa 3262 2876}%
\special{pa 3234 2876}%
\special{fp}%
\special{sh 1}%
\special{pa 3234 2876}%
\special{pa 3300 2896}%
\special{pa 3286 2876}%
\special{pa 3300 2856}%
\special{pa 3234 2876}%
\special{fp}%
%
\special{pn 8}%
\special{pa 2912 1680}%
\special{pa 2912 1268}%
\special{pa 1980 1268}%
\special{pa 1980 2874}%
\special{pa 1980 2874}%
\special{pa 1980 2874}%
\special{fp}%
%
\special{pn 8}%
\special{pa 1980 2874}%
\special{pa 2904 2874}%
\special{pa 2904 2874}%
\special{pa 2904 2874}%
\special{fp}%
%
\special{pn 8}%
\special{pa 2868 2874}%
\special{pa 2904 2874}%
\special{fp}%
\special{sh 1}%
\special{pa 2904 2874}%
\special{pa 2838 2854}%
\special{pa 2852 2874}%
\special{pa 2838 2894}%
\special{pa 2904 2874}%
\special{fp}%
%
\special{pn 8}%
\special{pa 2134 2804}%
\special{pa 2904 2804}%
\special{pa 2904 2804}%
\special{pa 2904 2804}%
\special{fp}%
%
\special{pn 8}%
\special{pa 2868 2804}%
\special{pa 2904 2804}%
\special{fp}%
\special{sh 1}%
\special{pa 2904 2804}%
\special{pa 2838 2784}%
\special{pa 2852 2804}%
\special{pa 2838 2824}%
\special{pa 2904 2804}%
\special{fp}%
%
\special{pn 8}%
\special{pa 2842 1686}%
\special{pa 2842 1582}%
\special{pa 2352 1582}%
\special{pa 2352 2736}%
\special{pa 2352 2736}%
\special{pa 2352 2736}%
\special{fp}%
%
\special{pn 8}%
\special{pa 2344 2736}%
\special{pa 2890 2736}%
\special{pa 2890 2736}%
\special{pa 2890 2736}%
\special{fp}%
%
\special{pn 8}%
\special{pa 2862 2736}%
\special{pa 2890 2736}%
\special{fp}%
\special{sh 1}%
\special{pa 2890 2736}%
\special{pa 2824 2716}%
\special{pa 2838 2736}%
\special{pa 2824 2756}%
\special{pa 2890 2736}%
\special{fp}%
%
\special{pn 8}%
\special{pa 2876 1680}%
\special{pa 2870 1442}%
\special{pa 2134 1442}%
\special{pa 2134 2806}%
\special{pa 2134 2806}%
\special{pa 2134 2806}%
\special{fp}%
%
\special{pn 8}%
\special{pa 2940 2666}%
\special{pa 2694 2666}%
\special{pa 2694 1756}%
\special{pa 2484 1756}%
\special{pa 2484 2702}%
\special{pa 2484 2694}%
\special{pa 2484 2694}%
\special{fp}%
%
\special{pn 8}%
\special{pa 2484 2710}%
\special{pa 2890 2710}%
\special{pa 2890 2710}%
\special{pa 2890 2710}%
\special{fp}%
%
\special{pn 8}%
\special{pa 2870 2710}%
\special{pa 2890 2710}%
\special{fp}%
\special{sh 1}%
\special{pa 2890 2710}%
\special{pa 2824 2690}%
\special{pa 2838 2710}%
\special{pa 2824 2730}%
\special{pa 2890 2710}%
\special{fp}%
\end{picture}%
}
\end{center}
\caption{$G(1,(1,0,3,0,\dots))$}
\end{figure}
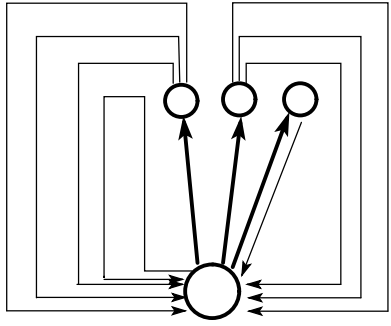
\begin{theorem}
For a finite directed graph $G=G(\mathcal V, \mathcal E)$ there exist data
$$  (R, (Q_M)_{M \in \Bbb N} )\in \Pi,$$ such that there is a topological conjugacy 
\begin{align*}
M\negthinspace {\scriptstyle D}({G})  \simeq M\negthinspace {\scriptstyle D}(G(  (R, (Q_M)_{M \in \Bbb N}  )), \tag {3.II.1}
\end{align*}
if and only if the associated semigroup of $M\negthinspace {\scriptstyle D}({G})$ is a Dyck inverse monoid, and 
\begin{align*}
\Lambda^{(e)} \leq 2, \qquad e \in \mathcal E \setminus \mathcal F _G, \tag {3.II.2}
\end{align*}
and in this case (3.1) holds for
\begin{align*}
R =& \card( \{ e \in  \mathcal E \setminus \mathcal F _G: \Lambda^{(e)} = 1\}), \tag {3.II.3}
\\
Q_M& =\tfrac{1}{M} \card (\{  e \in  \mathcal E \setminus \mathcal F _G: \Xi_4^{(e)}=
\\
&M+
I_2^{(0)}(M\negthinspace {\scriptstyle D}({G}) ) - \nu(M\negthinspace {\scriptstyle D}({G}) ) + \card(\mathcal M_1(G) ) \}), \quad M \in \Bbb N.
\end{align*}
\end{theorem}
\begin{proof}
The assumption, that the associated semigroup  of $M\negthinspace {\scriptstyle D}({G}))$ is a Dyck inverse monoid, implies, that 
the graph $G( \mathcal V , \mathcal F ) $ is a (possibly degenerate) tree.
 In  case that
 the tree $G( \mathcal V , \mathcal F ) $ is not degenerate, 
 one has from (3.II.2) that all leaves of the  subtree are at level 1. Also, for an $ e \in  \mathcal E \setminus \mathcal F _G $ , 
one has that 
 $$
 \card (\Xi^{(e)}_4)=\card (\{e^\prime  \in  \mathcal E \setminus \mathcal F _G  : s(e^\prime  ) = s(e)  \})+D(V_0).
 $$
 This follows from the observation, that in this case every orbit of length 4 with negative multiplier $e\in \mathcal E \setminus \mathcal F_G$ is obtained by inserting into an orbit of length 2 with negative multiplier $e\in \mathcal E \setminus \mathcal F_G$ either the word 
$f^-f^+, f \in \mathcal F_G, t(f) = s(e),$ or a word 
$\widetilde e^-\widetilde e^+, \widetilde e\in\mathcal M_1(G),$ or a word 
$\widetilde e^-\widetilde e^+,  \widetilde e\in\mathcal M_2(G), s(\widetilde e) = s(e)$. 
It is
$$
D(V_0) = I_2^{(0)}(M\negthinspace {\scriptstyle D}({G}))  - \nu(M\negthinspace {\scriptstyle D}({G}) ) + \card(\mathcal M_1(G) ) ,
$$
and the equations (3.II.3) follow.
\end{proof}
\begin{corollary}
For finite directed graphs $G=G(\mathcal V, \mathcal E)$ such that
  the associated semigroup of $M\negthinspace {\scriptstyle D}({G})$ is a Dyck inverse monoid, and such that
\begin{align*}
\Lambda^{(e)} \leq 2, \qquad e \in \mathcal E \setminus \mathcal F _G, 
\end{align*}
 the topological conjugacy of their Markov-Dyck shifts implies the isomorphim of the graphs.
\end{corollary}
\begin{proof}
In (3.II.3) the data $  (R, (Q_M)_{M \in \Bbb N} )$ are expressed in terms of invariants of topological conjugacy.
\end{proof}

\subsection{A family of finite directed graphs III}

We consider the family $\bold{F}_{III}$ of finite directed graphs 
$$
G[\ell, M] = G(\mathcal V[\ell, M],\mathcal E[\ell, M]), \qquad \ell , M \in \Bbb N, M \in \Bbb N,
$$
such that
$G(\mathcal V[\ell, M], \mathcal F_{G[\ell, M]} )$ is a non-degenerate tree.
We describe the graphs $G[\ell, M],\ell , M \in \Bbb N$ as follows: For
 $\ell , M \in \Bbb N$ the graph $G[\ell, M]$ has a vertex $V(0)$, vertices
$$
V_{0}(l), V_{1}(l),\qquad 1\leq l< \ell, 
$$
and  edges
$$
f_{0}(l), f_{1}(l),\qquad 1\leq l< \ell,  
$$
and
$$
e_{0}(m),e_{1}(m), \qquad 1 \leq m \leq M.
$$
The source and target mappings are given by
$$
s( f_0(1)  ) =  s( f_1(1)  )   = V(0),
$$
 \begin{multline*}
s(f_0(l)) = V_0(l - 1), \ s(f_1(l)) = V_1(l - 1), \
t(f_0(l)) = V_0(l),  \ t(f_1(l)) = V_1(l),  \\ 1 \leq l < \ell,
 \end{multline*}
and
 \begin{multline*}
s(e_0(m)) =  V_0(\ell -1), s(e_1(m))= V_1(\ell -1) , \ \
t(e_0(m)) =t(e_1(m))= V(0) ,\\ 1\leq m \leq M.
 \end{multline*}

The tree $G(\mathcal V[\ell, M], \mathcal F_{G[\ell, M]} )$
has the root $V_0$ and one has
$$
\mathcal F_{G[\ell, M]} = \{ f_0(l),f_1(l) ,1\leq l<\ell  \}, \quad \ell , M \in \Bbb N, M \in \Bbb N.
$$
Note the non-empty intersection of $ \bold{F}_{III}$ with  $ \bold{F}_{I}$and  $ \bold{F}_{II}$.
\begin{lemma}
For $G=G[\ell, M]$ one has that
$$
\tfrac{1}{2}I^{(0)}_2(M\negthinspace {\scriptstyle D}({G}))=\nu(M_D(G)) + 2 \Lambda
 (M\negthinspace {\scriptstyle D}({G})) -2.
$$
\end{lemma}
\begin{proof}

One has that
$$
 \card( \mathcal F_{G[\ell, M]})  = 2\Lambda (M\negthinspace {\scriptstyle D}({G[\ell, M]}))-2, \quad \ell , M \in \Bbb N, M \in \Bbb N. \qed
$$
\renewcommand{\qedsymbol}{}
\end{proof}
For $\ell, L, M \in \Bbb N$,
 $
 \ell \geq 4,   L < \ell - 2,
 $
 we introduce  
 auxiliary 
 graphs 
  $$
G_{2, M}[\ell, L] = G(\mathcal V_{2, M}[\ell, L]  , \mathcal E_{2, M}[\ell, L]  ), \quad 
G_{M, 2}[\ell, L] = G(\mathcal V_{M, 2}[\ell, L]  , \mathcal E_{M, 2}[\ell, L]   ).
 $$
  In both vertex sets $\mathcal V_{2, M}[\ell, L]  $   and $\mathcal V_{M, 2}[\ell, L]  $ there are vertices
 $$
 V(l), \qquad 0\leq l \leq L,
 $$
 and
 $$
 V_0(l,m), V_1(l,m), \qquad 1 \leq m \leq M,L + 2 \leq l < \ell,
 $$
 and in  both edge sets $\mathcal E_{2, M}[\ell, L]  $   and $\mathcal E_{M, 2}[\ell, L]  $ there are edges
 $$
 f(l), \qquad 0\leq l \leq L,
 $$
 and
 $$
 f_0(l,m), f_1(l,m), \qquad 1 \leq m \leq M,L + 2 \leq l < \ell,
 $$
 and
 $$
 e_0(m),  e_1(m),  \qquad 1 \leq m \leq M.
 $$
 with source and target vertices partially given by
 $$
 s( f(l) ) = V(l-1  ), \qquad 1\leq l \leq L,
 $$
 $$
  s(  f_0(l,m) ) = V_0(l-1,m  ), \ \   s(  f_1(l,m) ) = V_1(l-1,m  ), \qquad 1 \leq m \leq M, 
  L+1 <  l <\ell,
 $$
$$
 t( f(l)) = V(l), \qquad 1\leq l \leq L,
 $$
 $$
 t(f_0(l,m) ) = V_0(l,m  ), \ \   t(  f_1(l,m) ) = V_1(l,m  ), \qquad 1 \leq m \leq M, 
 L+1 < l  <  \ell,
 $$
 and
 \begin{multline*}
 s(e_0(m)) = V_0(\ell - 1  ,m  )  ,  s(e_1(m))= V_0(\ell - 1  ,m  ) , 
 t(e_0(m)) =   t(e_1(m)) = V(0), \\ 1 \leq m \leq M.
 \end{multline*}
 In addition, the graph $G_{2, M}[\ell, L]$ has vertices
 $$
 V_0(L+1), V_1(L+1), 
 $$
 and edges
 $$
 f_0(L+1), f_1(L+1),
 $$
 and the definition of its source and target mappings is completed by setting
 $$
 s(   f_0(L+1) )   =  s(f_1(L+1) )   =V(L),
 $$
 $$
 t(f_0(L+1) ) = V_0(L+1), \qquad  t(f_1(L+1) ) = V_1(L+1),
 $$
 and
 $$
 s(f_0(L+2,m) = V_0(L+1), \ \ s(f_1(L+2,m) = V_1(L+1),\qquad 1 \leq m \leq M.
 $$ 
 In addition, the graph $G_{M, 2}[\ell, L]$ has vertices
 $$
 V(L+1, m), \qquad 1 \leq m \leq M,
 $$
 and edges
 $$
 f (L+1, m), \qquad 1 \leq m \leq M,
 $$
 and the definition of its source and target mappings is completed by setting
 $$
 s( f (L+1, m) ) =  V(L) , \ \ t(  f (L+1, m)  ) =   V(L+1, m) , \qquad 1 \leq m \leq M,
 $$
 and
 $$
 s(f_0(L+2, m) =  s(f_1(L+2, m) =  V(L+1, m) , \qquad  1 \leq m \leq M.
 $$
 The graphs $G(\mathcal V_{2, M}[\ell, L]   ,  \mathcal F_{G_{2, M}[\ell, L]} )  $ and 
 $G( \mathcal V_{M, 2}[\ell, L]  , \mathcal F_{G_{2, M}[\ell, L]}  )   $ are directed trees.
 
\begin{figure}[htbp]
\begin{center}
{
\unitlength 0.1in
\begin{picture}( 28.0000, 21.7000)( 15.0000,-27.7000)
%
\special{pn 20}%
\special{ar 2900 2630 140 140  0.0000000 6.2831853}%
%
\special{pn 20}%
\special{pa 2900 2490}%
\special{pa 2900 2490}%
\special{fp}%
%
\special{pn 20}%
\special{pa 2900 2490}%
\special{pa 2900 2126}%
\special{fp}%
\special{sh 1}%
\special{pa 2900 2126}%
\special{pa 2880 2194}%
\special{pa 2900 2180}%
\special{pa 2920 2194}%
\special{pa 2900 2126}%
\special{fp}%
%
\special{pn 20}%
\special{pa 2952 1938}%
\special{pa 3232 1658}%
\special{fp}%
\special{sh 1}%
\special{pa 3232 1658}%
\special{pa 3172 1690}%
\special{pa 3194 1696}%
\special{pa 3200 1718}%
\special{pa 3232 1658}%
\special{fp}%
%
\special{pn 20}%
\special{pa 3376 1496}%
\special{pa 3564 1148}%
\special{fp}%
\special{sh 1}%
\special{pa 3564 1148}%
\special{pa 3514 1196}%
\special{pa 3538 1194}%
\special{pa 3550 1216}%
\special{pa 3564 1148}%
\special{fp}%
%
\special{pn 20}%
\special{pa 3250 1482}%
\special{pa 3064 1132}%
\special{fp}%
\special{sh 1}%
\special{pa 3064 1132}%
\special{pa 3078 1200}%
\special{pa 3090 1178}%
\special{pa 3114 1182}%
\special{pa 3064 1132}%
\special{fp}%
%
\special{pn 20}%
\special{pa 3310 1476}%
\special{pa 3316 1154}%
\special{fp}%
\special{sh 1}%
\special{pa 3316 1154}%
\special{pa 3296 1220}%
\special{pa 3316 1206}%
\special{pa 3336 1220}%
\special{pa 3316 1154}%
\special{fp}%
%
\special{pn 20}%
\special{pa 2838 1924}%
\special{pa 2558 1644}%
\special{fp}%
\special{sh 1}%
\special{pa 2558 1644}%
\special{pa 2590 1704}%
\special{pa 2596 1682}%
\special{pa 2618 1676}%
\special{pa 2558 1644}%
\special{fp}%
%
\special{pn 20}%
\special{pa 2424 1482}%
\special{pa 2218 1144}%
\special{fp}%
\special{sh 1}%
\special{pa 2218 1144}%
\special{pa 2236 1210}%
\special{pa 2246 1190}%
\special{pa 2270 1190}%
\special{pa 2218 1144}%
\special{fp}%
%
\special{pn 20}%
\special{pa 2550 1482}%
\special{pa 2736 1132}%
\special{fp}%
\special{sh 1}%
\special{pa 2736 1132}%
\special{pa 2686 1182}%
\special{pa 2710 1178}%
\special{pa 2722 1200}%
\special{pa 2736 1132}%
\special{fp}%
%
\special{pn 20}%
\special{pa 2492 1476}%
\special{pa 2484 1154}%
\special{fp}%
\special{sh 1}%
\special{pa 2484 1154}%
\special{pa 2466 1220}%
\special{pa 2486 1206}%
\special{pa 2506 1220}%
\special{pa 2484 1154}%
\special{fp}%
%
\special{pn 20}%
\special{ar 2900 2022 84 84  0.0000000 6.2831853}%
%
\special{pn 20}%
\special{ar 3320 1580 84 84  0.0000000 6.2831853}%
%
\special{pn 20}%
\special{ar 2508 1566 84 84  0.0000000 6.2831853}%
%
\special{pn 20}%
\special{ar 3600 1020 84 84  0.0000000 6.2831853}%
%
\special{pn 20}%
\special{ar 3320 1020 84 84  0.0000000 6.2831853}%
%
\special{pn 20}%
\special{ar 3048 1020 84 84  0.0000000 6.2831853}%
%
\special{pn 20}%
\special{ar 2746 1028 84 84  0.0000000 6.2831853}%
%
\special{pn 20}%
\special{ar 2488 1020 84 84  0.0000000 6.2831853}%
%
\special{pn 20}%
\special{ar 2200 1028 84 84  0.0000000 6.2831853}%
%
\special{pn 8}%
\special{pa 1920 880}%
\special{pa 2200 880}%
\special{fp}%
%
\special{pn 8}%
\special{pa 2200 930}%
\special{pa 2200 930}%
\special{fp}%
%
\special{pn 8}%
\special{pa 3600 880}%
\special{pa 3600 880}%
\special{fp}%
\special{pa 3600 916}%
\special{pa 3600 880}%
\special{fp}%
%
\special{pn 8}%
\special{pa 2200 880}%
\special{pa 2200 950}%
\special{fp}%
%
\special{pn 8}%
\special{pa 1920 2560}%
\special{pa 2726 2560}%
\special{pa 2726 2560}%
\special{fp}%
%
\special{pn 8}%
\special{pa 2698 2560}%
\special{pa 2726 2560}%
\special{fp}%
\special{sh 1}%
\special{pa 2726 2560}%
\special{pa 2658 2540}%
\special{pa 2672 2560}%
\special{pa 2658 2580}%
\special{pa 2726 2560}%
\special{fp}%
%
\special{pn 8}%
\special{pa 4300 600}%
\special{pa 4300 2700}%
\special{fp}%
%
\special{pn 8}%
\special{pa 1920 2560}%
\special{pa 1920 888}%
\special{fp}%
%
\special{pn 8}%
\special{pa 1710 740}%
\special{pa 1710 2630}%
\special{fp}%
%
\special{pn 8}%
\special{pa 1710 2630}%
\special{pa 2718 2630}%
\special{pa 2718 2630}%
\special{pa 2718 2630}%
\special{fp}%
%
\special{pn 8}%
\special{pa 2690 2630}%
\special{pa 2718 2630}%
\special{fp}%
\special{sh 1}%
\special{pa 2718 2630}%
\special{pa 2652 2610}%
\special{pa 2666 2630}%
\special{pa 2652 2650}%
\special{pa 2718 2630}%
\special{fp}%
%
\special{pn 8}%
\special{pa 4300 2698}%
\special{pa 3090 2698}%
\special{pa 3090 2698}%
\special{fp}%
%
\special{pn 8}%
\special{pa 3116 2698}%
\special{pa 3090 2698}%
\special{fp}%
\special{sh 1}%
\special{pa 3090 2698}%
\special{pa 3156 2718}%
\special{pa 3142 2698}%
\special{pa 3156 2678}%
\special{pa 3090 2698}%
\special{fp}%
%
\special{pn 8}%
\special{pa 1500 2700}%
\special{pa 2718 2700}%
\special{pa 2718 2700}%
\special{fp}%
%
\special{pn 8}%
\special{pa 2690 2700}%
\special{pa 2718 2700}%
\special{fp}%
\special{sh 1}%
\special{pa 2718 2700}%
\special{pa 2652 2680}%
\special{pa 2666 2700}%
\special{pa 2652 2720}%
\special{pa 2718 2700}%
\special{fp}%
%
\special{pn 8}%
\special{pa 1500 600}%
\special{pa 1500 2700}%
\special{fp}%
%
\special{pn 8}%
\special{pa 1500 600}%
\special{pa 2746 600}%
\special{pa 2746 930}%
\special{pa 2746 930}%
\special{pa 2746 930}%
\special{fp}%
%
\special{pn 8}%
\special{pa 1710 740}%
\special{pa 2480 740}%
\special{pa 2480 930}%
\special{pa 2480 936}%
\special{pa 2480 936}%
\special{fp}%
%
\special{pn 8}%
\special{pa 3600 880}%
\special{pa 3880 880}%
\special{fp}%
%
\special{pn 8}%
\special{pa 3880 2560}%
\special{pa 3880 888}%
\special{fp}%
%
\special{pn 8}%
\special{pa 3880 2560}%
\special{pa 3090 2560}%
\special{pa 3090 2560}%
\special{pa 3090 2560}%
\special{fp}%
%
\special{pn 8}%
\special{pa 3118 2560}%
\special{pa 3090 2560}%
\special{fp}%
\special{sh 1}%
\special{pa 3090 2560}%
\special{pa 3156 2580}%
\special{pa 3142 2560}%
\special{pa 3156 2540}%
\special{pa 3090 2560}%
\special{fp}%
%
\special{pn 8}%
\special{pa 3320 740}%
\special{pa 4090 740}%
\special{pa 4090 2630}%
\special{pa 4090 2630}%
\special{pa 4090 2630}%
\special{fp}%
%
\special{pn 8}%
\special{pa 4090 2630}%
\special{pa 3090 2630}%
\special{pa 3090 2630}%
\special{pa 3090 2630}%
\special{fp}%
%
\special{pn 8}%
\special{pa 3118 2630}%
\special{pa 3090 2630}%
\special{fp}%
\special{sh 1}%
\special{pa 3090 2630}%
\special{pa 3156 2650}%
\special{pa 3142 2630}%
\special{pa 3156 2610}%
\special{pa 3090 2630}%
\special{fp}%
%
\special{pn 8}%
\special{pa 4300 600}%
\special{pa 3054 600}%
\special{pa 3054 930}%
\special{pa 3054 930}%
\special{pa 3054 930}%
\special{fp}%
%
\special{pn 8}%
\special{pa 3320 740}%
\special{pa 3320 930}%
\special{fp}%
\end{picture}%
}
\end{center}
\caption{$G_{2,3}(4.1)$}
\end{figure}
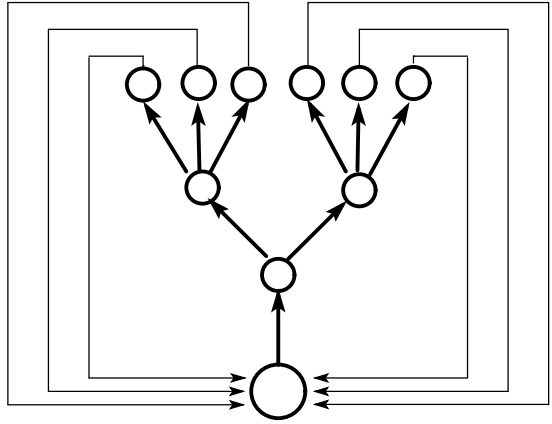


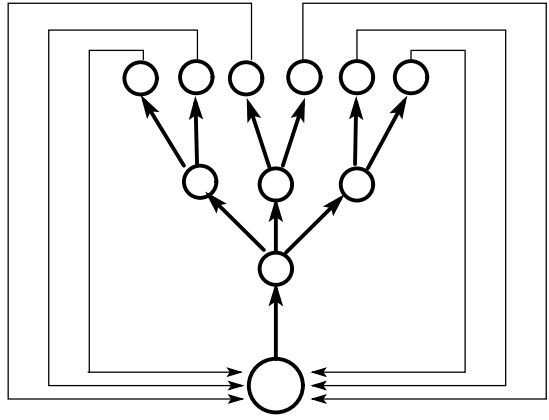
\begin{figure}[htbp]
\begin{center}
{\unitlength 0.1in
\begin{picture}( 27.8600, 21.3500)( 15.2000,-27.8500)
%
\special{pn 20}%
\special{ar 2906 2646 140 140  0.0000000 6.2831853}%
%
\special{pn 20}%
\special{pa 2906 2506}%
\special{pa 2906 2506}%
\special{fp}%
%
\special{pn 20}%
\special{pa 2906 2506}%
\special{pa 2906 2142}%
\special{fp}%
\special{sh 1}%
\special{pa 2906 2142}%
\special{pa 2886 2208}%
\special{pa 2906 2194}%
\special{pa 2926 2208}%
\special{pa 2906 2142}%
\special{fp}%
%
\special{pn 20}%
\special{pa 2958 1952}%
\special{pa 3238 1672}%
\special{fp}%
\special{sh 1}%
\special{pa 3238 1672}%
\special{pa 3178 1706}%
\special{pa 3200 1710}%
\special{pa 3206 1734}%
\special{pa 3238 1672}%
\special{fp}%
%
\special{pn 20}%
\special{pa 3382 1512}%
\special{pa 3570 1162}%
\special{fp}%
\special{sh 1}%
\special{pa 3570 1162}%
\special{pa 3520 1212}%
\special{pa 3544 1210}%
\special{pa 3556 1230}%
\special{pa 3570 1162}%
\special{fp}%
%
\special{pn 20}%
\special{pa 3316 1490}%
\special{pa 3322 1168}%
\special{fp}%
\special{sh 1}%
\special{pa 3322 1168}%
\special{pa 3302 1234}%
\special{pa 3322 1222}%
\special{pa 3342 1236}%
\special{pa 3322 1168}%
\special{fp}%
%
\special{pn 20}%
\special{pa 2844 1938}%
\special{pa 2564 1658}%
\special{fp}%
\special{sh 1}%
\special{pa 2564 1658}%
\special{pa 2596 1720}%
\special{pa 2602 1696}%
\special{pa 2624 1692}%
\special{pa 2564 1658}%
\special{fp}%
%
\special{pn 20}%
\special{pa 2430 1498}%
\special{pa 2224 1158}%
\special{fp}%
\special{sh 1}%
\special{pa 2224 1158}%
\special{pa 2242 1226}%
\special{pa 2252 1204}%
\special{pa 2276 1206}%
\special{pa 2224 1158}%
\special{fp}%
%
\special{pn 20}%
\special{pa 2498 1490}%
\special{pa 2490 1168}%
\special{fp}%
\special{sh 1}%
\special{pa 2490 1168}%
\special{pa 2472 1236}%
\special{pa 2492 1222}%
\special{pa 2512 1234}%
\special{pa 2490 1168}%
\special{fp}%
%
\special{pn 20}%
\special{ar 2906 2036 84 84  0.0000000 6.2831853}%
%
\special{pn 20}%
\special{ar 3326 1596 84 84  0.0000000 6.2831853}%
%
\special{pn 20}%
\special{ar 2514 1582 84 84  0.0000000 6.2831853}%
%
\special{pn 20}%
\special{ar 3606 1036 84 84  0.0000000 6.2831853}%
%
\special{pn 20}%
\special{ar 3326 1036 84 84  0.0000000 6.2831853}%
%
\special{pn 20}%
\special{ar 3054 1036 84 84  0.0000000 6.2831853}%
%
\special{pn 20}%
\special{ar 2752 1042 84 84  0.0000000 6.2831853}%
%
\special{pn 20}%
\special{ar 2494 1036 84 84  0.0000000 6.2831853}%
%
\special{pn 20}%
\special{ar 2206 1042 84 84  0.0000000 6.2831853}%
%
\special{pn 20}%
\special{ar 2906 1596 84 84  0.0000000 6.2831853}%
%
\special{pn 20}%
\special{pa 2906 1946}%
\special{pa 2906 1700}%
\special{fp}%
\special{sh 1}%
\special{pa 2906 1700}%
\special{pa 2886 1768}%
\special{pa 2906 1754}%
\special{pa 2926 1768}%
\special{pa 2906 1700}%
\special{fp}%
%
\special{pn 20}%
\special{pa 2872 1504}%
\special{pa 2766 1176}%
\special{fp}%
\special{sh 1}%
\special{pa 2766 1176}%
\special{pa 2768 1246}%
\special{pa 2782 1226}%
\special{pa 2806 1232}%
\special{pa 2766 1176}%
\special{fp}%
%
\special{pn 20}%
\special{pa 2942 1498}%
\special{pa 3046 1176}%
\special{fp}%
\special{sh 1}%
\special{pa 3046 1176}%
\special{pa 3006 1232}%
\special{pa 3030 1226}%
\special{pa 3044 1246}%
\special{pa 3046 1176}%
\special{fp}%
%
\special{pn 8}%
\special{pa 3606 944}%
\special{pa 3606 896}%
\special{pa 3886 896}%
\special{pa 3886 2576}%
\special{pa 3886 2576}%
\special{pa 3886 2576}%
\special{fp}%
%
\special{pn 8}%
\special{pa 3886 2576}%
\special{pa 3096 2576}%
\special{pa 3096 2576}%
\special{pa 3096 2576}%
\special{fp}%
%
\special{pn 8}%
\special{pa 3124 2576}%
\special{pa 3096 2576}%
\special{fp}%
\special{sh 1}%
\special{pa 3096 2576}%
\special{pa 3162 2596}%
\special{pa 3148 2576}%
\special{pa 3162 2556}%
\special{pa 3096 2576}%
\special{fp}%
%
\special{pn 8}%
\special{pa 3326 944}%
\special{pa 3326 790}%
\special{pa 4096 790}%
\special{pa 4096 2646}%
\special{pa 3096 2646}%
\special{pa 3096 2646}%
\special{pa 3096 2646}%
\special{fp}%
%
\special{pn 8}%
\special{pa 3124 2646}%
\special{pa 3096 2646}%
\special{fp}%
\special{sh 1}%
\special{pa 3096 2646}%
\special{pa 3162 2666}%
\special{pa 3148 2646}%
\special{pa 3162 2626}%
\special{pa 3096 2646}%
\special{fp}%
%
\special{pn 8}%
\special{pa 3046 944}%
\special{pa 3046 650}%
\special{pa 4306 650}%
\special{pa 4306 2716}%
\special{pa 3096 2716}%
\special{pa 3096 2716}%
\special{pa 3096 2716}%
\special{fp}%
%
\special{pn 8}%
\special{pa 3124 2716}%
\special{pa 3096 2716}%
\special{fp}%
\special{sh 1}%
\special{pa 3096 2716}%
\special{pa 3162 2736}%
\special{pa 3148 2716}%
\special{pa 3162 2696}%
\special{pa 3096 2716}%
\special{fp}%
%
\special{pn 8}%
\special{pa 2220 944}%
\special{pa 2220 896}%
\special{pa 1940 896}%
\special{pa 1940 2576}%
\special{pa 1940 2576}%
\special{pa 1940 2576}%
\special{fp}%
%
\special{pn 8}%
\special{pa 1934 2576}%
\special{pa 2724 2576}%
\special{pa 2724 2576}%
\special{pa 2724 2576}%
\special{fp}%
%
\special{pn 8}%
\special{pa 2696 2576}%
\special{pa 2724 2576}%
\special{fp}%
\special{sh 1}%
\special{pa 2724 2576}%
\special{pa 2658 2556}%
\special{pa 2672 2576}%
\special{pa 2658 2596}%
\special{pa 2724 2576}%
\special{fp}%
%
\special{pn 8}%
\special{pa 2500 944}%
\special{pa 2500 790}%
\special{pa 1730 790}%
\special{pa 1730 2646}%
\special{pa 2732 2646}%
\special{pa 2732 2646}%
\special{pa 2732 2646}%
\special{fp}%
%
\special{pn 8}%
\special{pa 2704 2646}%
\special{pa 2732 2646}%
\special{fp}%
\special{sh 1}%
\special{pa 2732 2646}%
\special{pa 2664 2626}%
\special{pa 2678 2646}%
\special{pa 2664 2666}%
\special{pa 2732 2646}%
\special{fp}%
%
\special{pn 8}%
\special{pa 2780 944}%
\special{pa 2780 650}%
\special{pa 1520 650}%
\special{pa 1520 2716}%
\special{pa 2732 2716}%
\special{pa 2732 2716}%
\special{pa 2732 2716}%
\special{fp}%
%
\special{pn 8}%
\special{pa 2704 2716}%
\special{pa 2732 2716}%
\special{fp}%
\special{sh 1}%
\special{pa 2732 2716}%
\special{pa 2664 2696}%
\special{pa 2678 2716}%
\special{pa 2664 2736}%
\special{pa 2732 2716}%
\special{fp}%
\end{picture}%
}
\end{center}
\caption{$G_{3.2}(4.1)$}
\end{figure}


The invariants of topological conjugacy $\nu$ and $\Lambda$, or, equivalently,  $\nu$ and $I^{(0)}_2$, together do not separate  the graphs  $G[\ell , M ]$ from the graphs  $G_{2,M}[\ell , L ]$  nor  from the graphs 
$G_{M,2}[\ell , L ]$, but together with the invariant 
$I^{(0)}_4$  they do, as the next lemma shows.
 
 \begin{lemma}

\indent
(a) For $\ell, M \in \Bbb N,$ and 
$
G = G[\ell , M ]
$
 one has
 $$
 I^{(0)}_4(M\negthinspace {\scriptstyle D}({G})) = 8\Lambda(M\negthinspace {\scriptstyle D}({G})) + \nu(M\negthinspace {\scriptstyle D}({G}))^2 + 10 \nu(M\negthinspace {\scriptstyle D}({G})) + 4.
 $$
 
 (b1) For $\ell, M \in \Bbb N, \ell > 3,  2 \leq L < \ell - 4,$ and 
$
G = G_{2,M}[\ell, L],
$
 one has
 \begin{multline*}
 I^{(0)}_4(M\negthinspace {\scriptstyle D}({G})) = 4\nu(M\negthinspace {\scriptstyle D}({G}))\Lambda(M\negthinspace {\scriptstyle D}({G})) + \nu(M\negthinspace {\scriptstyle D}({G}))^2 -2\nu(M\negthinspace {\scriptstyle D}({G})) + 8 \\
 +  
 4(1 - \nu(M\negthinspace {\scriptstyle D}({G})))L.
 \end{multline*}
 
 (b2) For $\ell, M \in \Bbb N, \ell > 3, 2 \leq L < \ell - 4,$ and 
$
G =  G_{M,2}[\ell, L]
$
 one has
 \begin{multline*}
 I^{(0)}_4(M\negthinspace {\scriptstyle D}({G})) = 4\nu(M\negthinspace {\scriptstyle D}({G}))\Lambda(M\negthinspace {\scriptstyle D}({G}))  + \tfrac{1}{2}\nu(M\negthinspace {\scriptstyle D}({G}))^2+
 3\nu(M\negthinspace {\scriptstyle D}({G}))-4 \\
 + 
 4(1 -  \nu(M\negthinspace {\scriptstyle D}({G})))L.
 \end{multline*}
\end{lemma}
\begin{proof} 
The number of neutral periodic orbits of length 4 of
 $$
M\negthinspace {\scriptstyle D}( G[\ell , M ]) \
(M\negthinspace {\scriptstyle D}(G_{2,M}[\ell , L ]),  M\negthinspace {\scriptstyle D}(G_{M,2}[\ell , L ])),
$$
that contain the points that carry the infinite concatenation of words of the form $e^-e^+\widetilde{e}^-\widetilde{e}^+, e \neq  \widetilde{e},$ is equal to
$$
1 +M(M-1) \
(1 +M(M-1), \tfrac{1}{2}M(M+1),
$$
and the number of neutral periodic orbits of length 4 of  $$
M\negthinspace {\scriptstyle D}( G[\ell , M ]) \
(  M\negthinspace {\scriptstyle D}(G_{2,M}[\ell , L ]),  M\negthinspace {\scriptstyle D}(G_{M,2}[\ell , L ])),
$$
 that contain the points that carry the infinite concatenation of words of the form
 $e^-\widetilde{e}^-\widetilde{e}^+ \widetilde{e}^-,t(e) = s(\widetilde{e})   ,$ is equal to
$$
2\ell + 6M -2 \
(2M\ell +2  - 2ML, 2\ell M + M - 1 - L). \qed
$$
\renewcommand{\qedsymbol}{}
\end{proof}

\begin{lemma}
For $\ell > 4$, and $M\in \Bbb N$ and for $G(\mathcal V, \mathcal E   )=G[\ell, M]$
 one has that
 \begin{align*}
\Xi^{(e)}_{\ell+ 2}(M\negthinspace {\scriptstyle D}({G}))=&\Lambda(M\negthinspace {\scriptstyle D}({G}))+
\tfrac{1}{2}\nu(M\negthinspace {\scriptstyle D}({G})), 
\\
\Xi^{(e)}_{\ell+ 4}(M\negthinspace {\scriptstyle D}({G}) )=&(\Lambda(M\negthinspace {\scriptstyle D}({G}) )
+
\tfrac{1}{2}\nu(M\negthinspace {\scriptstyle D}({G})) )^2 + 
\\ 
&\Lambda(M\negthinspace {\scriptstyle D}({G}) ) +
2\nu(M\negthinspace {\scriptstyle D}({G})) -2, 
\ \ e  \in \mathcal E \setminus \mathcal F_G.
\end{align*}
\end{lemma}
\begin{proof}
Let $ e  \in \mathcal E \setminus \mathcal F_G$, and let $ O^{(e)}$ be the shortest periodic orbit of $M\negthinspace {\scriptstyle D}({G})$ with negative multiplier $e $. 

All periodic orbits of 
$M\negthinspace {\scriptstyle D}({G}) $ of length $\Lambda(M\negthinspace {\scriptstyle D}({G}))+2$  with multiplier $e^- $
are obtained by  inserting a word of the form $g^-g^+$, where the source vertex of the edge $g^-$ is transversed by $ O^{(e)}$, into $ O^{(e)}$. The number of these words is 
$\ell + M$. 

All periodic orbits of 
$M\negthinspace {\scriptstyle D}({G}) $ of length $\Lambda(M\negthinspace {\scriptstyle D}({G}) )+4$  with multiplier $e^- $
are obtained by either inserting two words  of the form $g^-g^+$, where the source vertex of the edge $g^-$ is transversed by $ O^{(e)}$, into $ O^{(e)}$, or by inserting a word of the form 
$g^- \widetilde g^-\widetilde g^+g^+ , t(g) = s(\widetilde g),$ 
into $ O^{(e)}$, where the source vertex of the edge $g^-$ is transversed by $ O^{(e)}$, into 
$ O^{(e)}$, and the number of these words is $\ell + 4M-2$. 
\end{proof}

\begin{theorem}
For a finite directed graph $G(\mathcal V, \mathcal E)$ there exist 
$  \ell>4, M \in \Bbb N$, such that there is a topological conjugacy 
\begin{align*}
M\negthinspace {\scriptstyle D}({G})
 \simeq M\negthinspace {\scriptstyle D}(G[  \ell,  M  ]), \tag {3.III.1}
\end{align*}
if and only if 
there is a Dyck inverse monoid associated to $M\negthinspace {\scriptstyle D}({G}))$, all  $\Lambda(e), e \in \mathcal E \setminus \mathcal F_G,$ have the same value, and
\begin{align*}
 \tfrac{1}{2}I^{(0)}_2(M\negthinspace {\scriptstyle D}({G}))=2\Lambda (M\negthinspace {\scriptstyle D}({G}))+\nu(M\negthinspace {\scriptstyle D}({G}))-2, \tag A
\end{align*}
\begin{align*}
 I^{(0)}_4(M\negthinspace {\scriptstyle D}({G})) = 8\Lambda(M\negthinspace {\scriptstyle D}({G})) + \nu(M\negthinspace {\scriptstyle D}({G}))^2 + 10 \nu(M\negthinspace {\scriptstyle D}({G})) + 4,\tag B
\end{align*}
\begin{align*}
\Xi^{(e)}_{\ell+ 2}(M\negthinspace {\scriptstyle D}({G}))=\Lambda(M\negthinspace {\scriptstyle D}({G}) )+
\tfrac{1}{2}\nu(M\negthinspace {\scriptstyle D}({G})), \qquad e  \in \mathcal E \setminus \mathcal F_G,\tag C
\end{align*}
\begin{align*}
\Xi^{(e)}_{ \ell+4}(M\negthinspace {\scriptstyle D}({G}))=&(\Lambda(M\negthinspace {\scriptstyle D}({G}) ) 
+\tfrac{1}{2}\nu(M\negthinspace {\scriptstyle D}({G})) )^2 +  
\tag D 
\\
&\Lambda(M\negthinspace {\scriptstyle D}({G}) ) +
2\nu(M\negthinspace {\scriptstyle D}({G})) -2, 
\qquad e  \in \mathcal E \setminus \mathcal F_G.
\end{align*}

If conditions ${(A) ( B) (C) (D)}$ are satisfied, then (3.III.1) holds for
\begin{align*}
\ell = \Lambda (M\negthinspace {\scriptstyle D}({G})),\qquad 
M = \tfrac{1}{2}\nu(M\negthinspace {\scriptstyle D}({G})).  
\tag {3.III.2}
\end{align*}

\end{theorem}
\begin{proof}
Necessity follows from Lemma 3.5, Lemma 3.6. and Lemma 3.7. To prove sufficiency, let $G = G(\mathcal V, \mathcal E)$ be a graph that satisfies the conditions of the theorem. 
We denote the root of the 
tree $G(\mathcal V, \mathcal F_G)$ by $V_0$. The out-degree of a vertex we denote by $D$. It follows from (A) that 
$$
D(V_0) \leq 2.
$$
In the case  $D(V_0)= 2$, one has by (A) and (C) that the 
tree $G(\mathcal V, \mathcal F_G)$  has two leaves, that have the same out-degree, and equations (3.III.2) follow. The task is to exclude the case $D(V_0)= 1$.

Assume, that  $D(V_0)= 1$. Let $L$ be maximal, such that the 
tree 
$G(\mathcal V, \mathcal F_G)$ has a single vertex $V_L$  at level $L$. One has that
$$
L < \Lambda(M\negthinspace {\scriptstyle D}({G})) - 2,
$$
since otherwise by (A),
$$
D(V_L) =  \Lambda(M\negthinspace {\scriptstyle D}({G})),
$$
which is either by (C) incompatible with $L > 0$, or it contradicts (B).

 By deriving a contradiction to (A)(B)(C)(D) we will exclude each of the following cases
 (c1 - 5):

\begin{align*}
D(V_L) >  \tfrac{1}{2}\nu(M\negthinspace {\scriptstyle D}({G})) +1, \tag {c1}
\end{align*}
\begin{align*}
D(V_L)=  \tfrac{1}{2}\nu(M\negthinspace {\scriptstyle D}({G})) +1, \tag {c2}
\end{align*}
\begin{align*}
D(V_L) =\tfrac{1}{2}\nu(M\negthinspace {\scriptstyle D}({G})), \tag {c3}
\end{align*}
\begin{align*}
 \tfrac{1}{2}\nu(M\negthinspace {\scriptstyle D}({G})) > D(V_L) > 2, \tag {c4}
\end{align*}
\begin{align*}
D(V_L)=2. \tag {c5}
\end{align*}

We consider cycles
$$
b = (e_k)_{1 \leq k\leq  \Lambda(M\negthinspace {\scriptscriptstyle D}({G}))}
$$
in $G$, such that 
$
s(e_1) = V_L.
$
By (C)
\begin{align*}
\sum_{1 \leq k\leq  \Lambda(M\negthinspace {\scriptscriptstyle D}({G}))}  ( D(s(e_k)) - 1) = \tfrac{1}{2}\nu(M\negthinspace {\scriptstyle D}({G})). \tag           {3.III.3}
\end{align*}

In case (c1) one has a contradiction to (3.III.3) and therefore to (C).

Case (c2) is by (C) only possible if
$$
\nu(M\negthinspace {\scriptstyle D}({G})) = D(V_L) = 2,
$$
which contradicts (A).

In case (c3)  it follows from (3.III.3) for the cycle 
$b= (e_k)_{1 \leq k\leq  \Lambda(M\negthinspace {\scriptscriptstyle D}({G}))}$, that
\begin{align*}
\sum_{1 < k\leq  \Lambda(M\negthinspace {\scriptscriptstyle D}({G}))}  ( D(s(e_k)) - 1) = 1,
\end{align*}
and by (D) the only other vertex besides $s(e_1)$, that is traversed by $b$, that has an out-degree, that exceeds one (and is equal to two), is necessarily $s(e_2)$. This means that $G$ is isomorphic to $G_{\tfrac{1}{2}\nu(M\negthinspace {\scriptscriptstyle D}({G})),2}(\Lambda(M\negthinspace {\scriptstyle D}({G})), L) $, and by Lemma 3 (A) and (B) yield a contradiction.

For case (c4) we set
$$
d = D(V_L) - 2,
$$
and we have from (3.III.3) that
\begin{align*}
\sum_{1 < k\leq 
 \Lambda(M\negthinspace {\scriptscriptstyle D}({G}))}  ( D(s(e_k)) - 1) =\tfrac{1}{2}\nu(M\negthinspace {\scriptstyle D}({G}))  - 1 - d. \tag{3.III.4}
\end{align*}
One has
\begin{align*}
(2 + d)(\tfrac{1}{2}\nu(M\negthinspace {\scriptstyle D}({G})) - d) > \nu(M\negthinspace {\scriptstyle D}({G})), \qquad 0 < d < \tfrac{1}{2}\nu(M\negthinspace {\scriptstyle D}({G}))  - 2.   \tag {3.III.5}
\end{align*}
It follows from (3.III.4) that 
$$
\nu(M\negthinspace {\scriptstyle D}({G})) \geq D(V_L)( \tfrac{1}{2}\nu(M\negthinspace {\scriptstyle D}({G})) - d),
$$
which contradicts (3.III.5).

With
$
G_{2,\tfrac{1}{2}\nu(M\negthinspace {\scriptscriptstyle D}({G}))}(\Lambda(M\negthinspace {\scriptstyle D}({G})), L)
$
in place of $G_{\tfrac{1}{2}\nu(M\negthinspace {\scriptscriptstyle D}(G)),2}(\Lambda(
M\negthinspace {\scriptstyle D}({G})
), L) $ and with statement (b2) of Lemma 3.6 in place of statement (b1), one has for case (c5)  the same argument as for case (c3).
\end{proof}
\begin{corollary}
For directed graphs $G(\mathcal V, \mathcal E )$, such that  $\mathcal S(M\negthinspace {\scriptstyle D}(G(\mathcal V, \mathcal E )))$  is  a Dyck inverse monoid, and
that satisfy conditions (A)(B)(C)(D), the topological  conjugacy of the Markov-Dyck shifts $M\negthinspace {\scriptstyle D}(G(\mathcal V  , \mathcal E ) )$ implies the isomorphism of the graphs $G(\mathcal V  , \mathcal E )$.
\end{corollary}
\begin{proof}
In (3.III.2) the data $ [\ell,M]$ are expressed in terms of invariants of topological conjugacy.
\end{proof}

\medskip

\par\noindent Wolfgang Krieger
\par\noindent Institut f\"ur Angewandte Mathematik, 
\par\noindent  Universit\"at Heidelberg,
\par\noindent Im Neuenheimer Feld 205, 
 \par\noindent 69120 Heidelberg,
 \par\noindent Germany
 \par\noindent krieger@math.uni-heidelberg.de

\bigskip

\par\noindent Kengo Matsumoto
\par\noindent  Department of Mathematics,
\par\noindent  Joetsu University of Education, 
 \par\noindent Joetsu 943 - 8512,
 \par\noindent Japan
\par\noindent kengo@juen.ac.jp

 \end{document}